\documentclass{amsart}
\usepackage{amsmath,amsthm,amssymb}
\usepackage{dsfont}
\usepackage[all,cmtip]{xy}
\input{diagxy}

\usepackage[cm]{fullpage}

 \theoremstyle{plain}
 \newtheorem{thm}{Theorem}[section]
 \newtheorem{cor}[thm]{Corollary}
 \newtheorem{lem}[thm]{Lemma}
 
 \newtheorem{prop}[thm]{Proposition}

\theoremstyle{definition}
 \newtheorem{defn}[thm]{Definition}

\theoremstyle{remark}
 \newtheorem{rem}[thm]{Remark}

 \newtheorem{nota}[thm]{Notation}
 \newtheorem{conv}[thm]{Convention}

 \numberwithin{equation}{section}

\DeclareMathOperator{\VF}{VF} \DeclareMathOperator{\ACVF}{ACVF}
 
\DeclareMathOperator{\RV}{RV}

 \DeclareMathOperator{\ran}{ran}
 \DeclareMathOperator{\dom}{dom}

 \DeclareMathOperator{\id}{id}

 \DeclareMathOperator{\lh}{lh}

 \DeclareMathOperator{\acl}{acl}
 \DeclareMathOperator{\dcl}{dcl}
 \DeclareMathOperator{\pr}{pr}
 \DeclareMathOperator{\alg}{ac}

\DeclareMathOperator{\jcb}{Jcb}

\DeclareMathOperator{\K}{\overline{K}}

\def\Xint#1{\mathchoice
{\XXint\displaystyle\textstyle{#1}}%
{\XXint\textstyle\scriptstyle{#1}}%
{\XXint\scriptstyle\scriptscriptstyle{#1}}%
{\XXint\scriptscriptstyle\scriptscriptstyle{#1}}%
\!\int}
\def\XXint#1#2#3{{\setbox0=\hbox{$#1{#2#3}{\int}$}
\vcenter{\hbox{$#2#3$}}\kern-.5\wd0}}

\newcommand{\Z}{\mathds{Z}}

\newcommand{\N}{\mathds{N}}

\newcommand{\p}{$p$\nobreakdash}
\newcommand{\omin}{$o$\nobreakdash}

\newcommand{\cmin}{$C$\nobreakdash}

\newcommand{\gA}{\mathfrak{A}}
\newcommand{\gB}{\mathfrak{B}}
\newcommand{\gC}{\mathfrak{C}}

\newcommand{\gM}{\mathfrak{M}}

\newcommand{\ga}{\mathfrak{a}}
\newcommand{\gb}{\mathfrak{b}}
\newcommand{\gc}{\mathfrak{c}}
\newcommand{\gd}{\mathfrak{d}}

\newcommand{\gh}{\mathfrak{h}}

\newcommand{\gl}{\mathfrak{l}}
\newcommand{\gm}{\mathfrak{m}}

\newcommand{\go}{\mathfrak{o}}
\newcommand{\gp}{\mathfrak{p}}
\newcommand{\gq}{\mathfrak{q}}

\newcommand{\gx}{\mathfrak{x}}
\newcommand{\gy}{\mathfrak{y}}
\newcommand{\gz}{\mathfrak{z}}

\newcommand{\0}{\emptyset}

 \newcommand{\abs}[1]{\left\vert#1\right\vert}
 
 \newcommand{\set}[1]{\left\{#1\right\}}
 \newcommand{\seq}[1]{\left<#1\right>}

 \newcommand{\wh}[1]{\widehat{#1}}
 \newcommand{\lan}[1]{\mathcal{L}_{\textup{#1}}}

\newcommand{\mdl}[1]{\mathcal{#1}}  
\newcommand{\bb}[1]{\mathbb{#1}}

\newcommand{\rest}{\upharpoonright}
\newcommand{\lbar}{\vec}

\newcommand{\fun}{\longrightarrow}
\newcommand{\efun}{\longmapsto}
\newcommand{\sub}{\subseteq}

\newcommand{\mi}{\smallsetminus}

\newcommand{\colim}[1]{\underset{#1}{\text{colim}}}  

\newcommand{\la}{\langle}
\newcommand{\ra}{\rangle}

\DeclareMathOperator{\mVF}{\mu \! \VF}
\DeclareMathOperator{\mgVF}{\mu_{\Gamma} \! \VF}
\DeclareMathOperator{\mRV}{\mu \! \RV}
\DeclareMathOperator{\mgRV}{\mu_{\Gamma} \! \RV}
\DeclareMathOperator{\rv}{rv}

\DeclareMathOperator{\vv}{val}

\DeclareMathOperator{\gsk}{\mathbf{K}_+}
\DeclareMathOperator{\ggk}{\mathbf{K}}
\DeclareMathOperator{\ob}{Ob}
\DeclareMathOperator{\fn}{FN}
\DeclareMathOperator{\fib}{fib}

\DeclareMathOperator{\isp}{I_{sp}}
\DeclareMathOperator{\misp}{\mu I_{sp}}
\DeclareMathOperator{\ispt}{I_{sp}^{\times}}

\DeclareMathOperator{\rad}{rad}
\DeclareMathOperator{\vcr}{vcr}
\DeclareMathOperator{\vrv}{vrv}
\DeclareMathOperator{\RVH}{RVH}

\DeclareMathOperator{\can}{\mathbf{c}}

\DeclareMathOperator{\pvf}{pvf}
\DeclareMathOperator{\prv}{prv}

\DeclareMathOperator{\KR}{\mathbf{K} \mathds R}

\begin{document}

\title[Integration in ACVF]{Integration in algebraically closed valued fields}

\author{Yimu Yin}

\address{Department of Mathematics, University of Pittsburgh, 301 Thackeray Hall, Pittsburgh, PA  15260}

\email{yimuyin@pitt.edu}

\begin{abstract}
The first two steps of the construction of motivic integration in the fundamental work of Hrushovski and
Kazhdan~\cite{hrushovski:kazhdan:integration:vf} have been presented in~\cite{Yin:special:trans}. In this paper we present the final third step. As in~\cite{Yin:special:trans}, we limit our attention to the theory of algebraically closed valued fields of pure characteristic $0$ expanded by a $(\VF,
\Gamma)$-generated substructure $S$ in the language $\lan{RV}$. A canonical description of the kernel of the homomorphism $\bb L$ is obtained.
\end{abstract}

\maketitle

\tableofcontents

\section{Introduction}

To describe in a few words how motivic integration is different from classical integration it seems best to begin by pointing out that the ring that provides values for integrals is not the real field but a Grothendieck ring. The latter is traditionally constructed from equivalence classes of algebraic varieties and, more generally in the model-theoretic setting, from equivalence classes of definable subsets. Topological tools that are essential to many classical constructions are no longer available; instead, since it was first introduced by M. Kontsevich in 1995, techniques from first-order model theory of definable sets underlie much of the development of this new kind of integration. In fact, at risk of being overly simple-minded, one may think of motivic integration as classical integration with the topological concepts of ``continuity'', ``convergence'', etc.\ replaced everywhere by the model-theoretic concept of ``definability''.

To be sure, the class of definable integrals is conceptually narrower than the class of integrals that can be more or less dealt with classically. However, there are many reasons why the motivic approach to integration will play an increasingly important role. We mention two here.

Firstly, the progress in model theory in the last few decades suggests that many natural mathematical properties are subject to first-order treatment. In our context, given the fact that some very complicated integral identities are already motivic (see, for example, \cite{cluckers:cunningham:gordon:spice, cluckers:hales:loeser:transfer, cunning:hales:good}), it is reasonable to expect that many other important kinds of integrals are definable in some first-order languages and hence may be studied motivically. We note that, in their recent paper~\cite{hru:kazh:2009}, Hrushovski and Kazhdan have developed a partially first-order method to study adelic structures over curves and, in particular, have obtained a global Poisson summation formula.

Secondly, if one is more interested in the structure of a space of functions (for example, functional equations) than actual computation of functions, then constantly worrying about things such as convergence seems to be an unnecessary burden. By this we just mean that there is no need to insist on assigning ``numerical values'' to integrals, especially when it is not possible, and sometimes working with ``geometrical values'' is more effective. Definable integrals are of a more geometrical nature and are better behaved, at least before specializing to local fields. Some pathological phenomena afforded by point-set topology are thus avoided. For example, while classically it is possible that two iterative integrals of a function exist but are not equal, this cannot happen to definable integrals. This is our Fubini's theorem (Theorem~\ref{fubini}).

The Hrushovski-Kazhdan integration theory~\cite{hrushovski:kazhdan:integration:vf} is a major development in the theory of motivic integration. The fundamental idea is to construct homomorphisms between various Grothendieck rings associated with the first-order theory $\ACVF_S(0,0)$ of algebraically closed valued fields as naturally formulated in the language $\lan{RV}$. This construction has three main steps, which is described in the introduction of~\cite{Yin:special:trans}. For clarity we shall briefly recall what they are here. Let $\mVF_*$ and $\mRV[*]$ be two
categories of definable sets with volume forms that are respectively associated with the $\VF$-sort and the $\RV$-sort of $\lan{RV}$ (see~\cite[Section~10]{Yin:special:trans}). To construct a canonical
homomorphism from the Grothendieck semigroup $\gsk \mVF_*$ to the Grothendieck semigroup $\gsk \mRV[*] / \misp$, where $\misp$ is a suitable semigroup congruence relation, we proceed as follows:
\begin{itemize}
 \item\emph{Step~1.} There is a natural lifting map $\bb L$ from the set of objects of $\mRV[*]$ into the set of objects of $\mVF_*$. Using only special bijections, we show that $\bb L$ hits every isomorphism class of $\mVF_*$. The reason that we want to use only special bijections is that they are essentially compositions of additive translations and hence are measure-preserving. This is crucial for Step~3 below.

 \item\emph{Step~2.} We show that $\bb L$ induces a semigroup homomorphism from $\gsk \mRV[*]$ into $\gsk \mVF_*$, which is also denoted by $\bb L$.

 \item\emph{Step~3.} We obtain a precise description of the semigroup congruence relation on $\gsk \mRV[*]$ induced by $\bb L$ using the basic notion of a blowup of an object in $\mRV[*]$. The basic idea is that, for any objects $\mathbf{U}_1, \mathbf{U}_2$ in $\mRV[*]$, there are isomorphic iterated blowups $\mathbf{U}_1^{\sharp}$, $\mathbf{U}_2^{\sharp}$ of $\mathbf{U}_1$, $\mathbf{U}_2$ if and only if $\bb L(\mathbf{U}_1)$, $\bb L(\mathbf{U}_2)$ are isomorphic. The ``if'' direction essentially contains a form of Fubini's Theorem and is the most technically involved part of the construction.
\end{itemize}
When the Grothendieck semigroups are formally groupified, the inverse of $\bb L$ thus obtained is recast as a ring homomorphism (Theorem~\ref{theorem:integration:1:vol}). This is understood as a motivic integration, since Fubini's theorem (Theorem~\ref{fubini}) and the change of variables formula (Theorem~\ref{change:variables}) hold.

The first two steps are presented in~\cite{Yin:special:trans}. In this paper we present the third step. We mention again what has been said in the introduction of~\cite{Yin:special:trans} that, conceptually, the Hrushovski-Kazhdan construction of motivic integration is completed by the first two steps. Besides the obvious benefit of a deeper understanding, the main reason that we would like to carry out the third step is to facilitate computation in future applications.

The sections are organized as follows. Throughout this paper we shall use the notation and terminology of~\cite{Yin:QE:ACVF:min, Yin:special:trans}, some of which are recalled in Section~\ref{section:review}, where we also review the results in~\cite{Yin:QE:ACVF:min, Yin:special:trans} that shall be cited later. In Section~\ref{section:atoms} we prove some structural properties concerning the basic geography of definable sets in $\ACVF_S(0, 0)$. The main goal is to show that definable bijections between subsets of $\VF$ have the open-to-open property over a finite partition into definable subsets (Proposition~\ref{partition:open:to:open}). The setting of Section~\ref{section:atoms} is a bit different from the corresponding sections in~\cite{Yin:QE:ACVF:min, Yin:special:trans}: we work
at the level of types with imaginary parameters, namely balls, allowed.

The notion of a $2$-cell is introduced in Section~\ref{section:2unit}, which corresponds to the notion of a bicell in \cite{cluckers:loeser:constructible:motivic:functions}. This notion may look strange and is, perhaps, only of technical
interest. It arises when we try to prove some form of Fubini's
theorem, such as Lemma~\ref{contraction:perm:pair:isp}. The
difficulty is that, although, using \cmin-minimality, integrating definable sets of
$\VF$-dimension $1$ is very functorial (see Lemma~\ref{simul:special:dim:1}), we are unable to extend the construction to higher $\VF$-dimensions. This is the concern of \cite[Question~7.9]{hrushovski:kazhdan:integration:vf}. It has
also occurred in \cite{cluckers:loeser:constructible:motivic:functions} and may be
traced back to \cite{dries:1998}; see
\cite[Section~1.7]{cluckers:loeser:constructible:motivic:functions}.
Anyway, in this situation, the natural strategy of integrating
definable sets of higher $\VF$-dimensions is to use the result for
$\VF$-dimension $1$ and proceed with one $\VF$-sort variable at a time. As in the classical theory of integration, this strategy requires some form of Fubini's theorem: for a well-behaved integration, an integral should give the same value
when it is evaluated along different orders of variables. By induction, this problem is immediately reduced to the case of two variables. A $2$-cell is a definable
subset of $\VF^2$ with certain symmetrical geometrical structure that satisfies this Fubini type of requirement. Now the idea is that, if we can find a definable partition for every definable subset such that each piece is a $2$-cell indexed by some $\RV$-sort parameters, then, by compactness, every definable
subset satisfies the Fubini type of requirement. This kind of
partition is achieved in Lemma~\ref{decom:into:2:units}.

The key result of Section~\ref{section:contracting}, Lemma~\ref{simul:special:dim:1}, says that, modulo special bijections, every definable bijection between two definable sets of $\VF$-dimension $1$ is equal to the lift of an isomorphism in the corresponding $\RV$-category. As has just been remarked above, it would be ideal to extend this result to definable sets of all $\VF$-dimensions. Failing that, we will have to go down a path that is full of undesirable technicalities (including the whole discussion on $2$-cells). We introduce the notion of a standard contraction, which gives rise to the Fubini type of problem described above; see Definition~\ref{defn:standard:contraction}.
Then in Lemma~\ref{2:unit:contracted} we show that an essential
part of Lemma~\ref{simul:special:dim:1} holds for $2$-cells, which
is good enough for the rest of the construction.

The task of finding a canonical description of the kernel of $\bb L$, that is, Step~3 above, is carried out in Section~\ref{section:kernel} for the categories without volume forms. We introduce the notion of a blowup and then show that the equivalence relation $\isp[*, \cdot]$ it induces on $\RV[*, \cdot]$ is indeed a semigroup congruence relation; see Definition~\ref{defn:blowup} and Lemma~\ref{isp:congruence}. The key result of this section is Proposition~\ref{kernel:L}, which says that $\isp[*, \cdot]$ is the congruence relation induced by the homomorphism $\bb L$. With this in hand, we can describe canonically the isomorphisms between the various Grothendieck semigroups given by the inversion of $\bb L$. Also, because we know how the kernel of $\bb L$ is generated, there is a very elegant injective ring homomorphism (Theorem~\ref{theorem:integration:1}).

In the last section we basically repeat the work in Section~\ref{section:kernel} for the categories with volume forms. Additional arguments are needed at a few places since we now work with stricter morphisms (measure-preserving maps), for example, see Lemma~\ref{elementary:blowups:preserves:iso:vol}. In the end we obtain Fubini's theorem and the change of variable formula.

We note that at the level of local fields there is another very general approach to motivic integration, namely the Cluckers-Loeser theory~\cite{cluckers:loeser:constructible:motivic:functions}; see~\cite{gordon:yaffe:2008} for an excellent exposition. However, in a future paper we shall show that the Hrushovski-Kazhdan theory and the Cluckers-Loeser theory are compatible (via specialization of the former to the latter) and hence draw the relieving conclusion that there is really just one theory of motivic integration.

\section{Preliminaries and review}\label{section:review}

The reader is referred to~\cite{Yin:QE:ACVF:min, Yin:special:trans} for notation and terminology. Some of them will be recalled in the discussion below.

In the language $\lan{RV}$, the two sorts $\VF$ and $\RV$ without the zero elements are respectively denoted by $\VF^{\times}$ and $\RV$, and $\RV \mi \set{\infty}$ is denoted by $\RV^{\times}$. The set $\{ x \in \RV: x > 1 \}$ is denoted by $\RV^{> 1}$.

Since a $\VF$-sort literal can be equivalently expressed as an $\RV$-literal, we may assume that an $\lan{RV}$-formula contains no $\VF$-sort literals at all. In particular, we may assume that every $\VF$-sort polynomial $F(\lbar X)$ in a formula $\phi$ occurs in the form $\rv(F(\lbar X))$. This understanding sometimes makes the discussion more streamlined. We say that $F(\lbar X)$ is an \emph{occurring polynomial} of $\phi$.

Except in Section~\ref{section:atoms}, as in~\cite{Yin:QE:ACVF:min, Yin:special:trans}, we shall work in a sufficiently saturated model $\gC$ of the theory $\ACVF_S(0, 0)$, where $S \sub \gC$ is a small substructure such that \emph{$\Gamma(S)$ is nontrivial}.

\begin{conv}
By a definable subset of $\gC$ we mean a $\0$-definable
subset in the theory $\ACVF_S(0, 0)$. If additional parameters are used in defining a subset then we shall spell them out explicitly if necessary. As in~\cite{Yin:special:trans}, elements in the imaginary sort $\Gamma$ are frequently used as parameters in formulas.
\end{conv}

The substructure generated by a subset $A$ is denoted by $\la A \ra$ or $\dcl(A)$. The model-theoretic algebraic closure of $A$ is denoted by $\acl(A)$, whereas the field-theoretic algebraic closure of $A$, if applicable, is denoted by $A^{\alg}$. A substructure $S$ is \emph{$\VF$-generated} if there is a subset $A \sub \VF(S)$ such that $S  = \la A \ra$; similarly for $(\VF, \RV)$-generated substructures, $(\VF, \Gamma)$-generated substructures, etc.

\begin{defn}
A subset $\gb$ of $\VF$ is an \emph{open ball} if there is a
$\gamma \in \Gamma$ and a $b \in \gb$ such that $a \in \gb$ if and
only if $\vv(a - b) > \gamma$. It is a \emph{closed ball} if $a
\in \gb$ if and only if $\vv(a - b) \geq \gamma$. It is an
\emph{$\rv$-ball} if $\gb = \rv^{-1}(t)$ for some $t \in \RV$. The
value $\gamma$ is the \emph{radius} of $\gb$, which is denoted as
$\rad (\gb)$. Each point in $\VF$ is a closed ball of radius $\infty$ and $\VF$ is a clopen ball of radius $- \infty$.

If $\vv$ is constant on $\gb$ --- that is, $\gb$ is contained in an $\rv$-ball --- then $\vv(\gb)$ is the \emph{valuative center} of $\gb$; if $\vv$ is not constant on
$\gb$, that is, $0 \in \gb$, then the \emph{valuative center} of
$\gb$ is $\infty$. The valuative center of $\gb$ is denoted by
$\vcr(\gb)$.

A subset $\gp \sub \VF^n \times \RV^m$ is an (\emph{open, closed, $\rv$-}) \emph{polydisc} if it is of the form $(\prod_{i \leq n} \gb_i) \times \set{\lbar t}$, where each $\gb_i$ is an (open, closed, $\rv$-) ball and $\lbar t \in \RV^m$.
If $\gp$ is a polydisc then the \emph{radius} of $\gp$, denoted as $\rad(\gp)$,
is $\min \set{\rad(\gb_i) : i \leq n}$. The open and closed polydiscs centered at a sequence of elements $\lbar a = (a_1, \ldots, a_n) \in \VF^n$ with radii $\lbar \gamma = (\gamma_1, \ldots, \gamma_n) \in \Gamma^n$ are respectively denoted as $\go(\lbar a, \lbar \gamma)$ and $\gc(\lbar a, \lbar \gamma)$.

An $\rv$-polydisc $\rv^{-1}(t_1, \ldots, t_n) \times \{ \lbar s \}$ is \emph{degenerate} if $t_i = \infty$ for some $i$.
\end{defn}

\begin{defn}
A subset $\gd$ of $\VF$ is a \emph{punctured} (\emph{open, closed,
$\rv$-}) \emph{ball} if $\gd = \gb \mi \bigcup_{i=1}^n \gh_i$, where
$\gb$ is an (open, closed, $\rv$-) ball, $\gh_i, \ldots, \gh_n$
are disjoint balls, and $\gh_i, \ldots, \gh_n \sub \gb$. Each
$\gh_i$ is a \emph{hole} of $\gd$. The \emph{radius} and the
\emph{valuative center} of $\gd$ are those of $\gb$. A punctured closed ball $\gb$ with a single hole $\gh$ such that $\gh$ is a maximal open subball of $\gb$ is called a \emph{thin annulus}.
\end{defn}

For example, an element $\gamma \in \Gamma$ may be regarded as a thin annulus: it
is the punctured closed ball centered at $0$ with radius $\gamma$ and with the maximal open subball containing $0$ removed.

For brevity, we introduce the following peculiar terminology: by ``a punctured ball $\gb$'' we shall mean that $\gb$ is a ball that is possibly not punctured at all (or punctured by the ``empty ball''). So by ``a ball'' we shall always mean an unpunctured ball.

\begin{rem}\label{rem:K:aff}
Given two balls $\ga$, $\gb$, we write $\ga - \gb$ for the subset $\{a - b : a \in \ga \text{ and } b \in \gb \}$. Suppose that $\ga$, $\gb$ are disjoint maximal open subballs of a closed ball $\gc$. Then clearly $\ga - \gb$ is an $\rv$-ball $\rv^{-1}(t)$ with $\vrv(t) = \rad(\ga)$. This means that the collection of maximal open subballs of $\gc$ admits a $\K$-affine structure.
\end{rem}

In Section~\ref{section:atoms} we shall use balls as parameters as well. We could work in the traditional expansion $\gC^{eq}$ of $\gC$. But a much simpler expansion $\gC^{\bullet}$ suffices: it has only one additional sort that contains all the open balls and all the closed balls whose valuative radii are in $\Gamma(S)$. This means that, when we work in $\gC^{\bullet}$, the underlying substructure $S$ actually contains balls as imaginary elements. This expansion can help reduce the technical complexity of our discussion. However, as is the case with $\Gamma$, it is conceptually inessential since, for the purposes of this paper, all allusions to balls as imaginary elements may be eliminated in favor of objects already definable in $\gC$.

\begin{nota}
The following notational device is used in $\gC^{\bullet}$: for a ball $\gb \sub \VF$, if there is a corresponding imaginary element then we shall denote it by $\dot \gb$.
\end{nota}

We shall adopt~\cite[Convention~4.20]{Yin:QE:ACVF:min}: Since definably bijective subsets are to be identified, for a subset $A$, we shall tacitly substitute its canonical image $\can(A)$ for it in the discussion if it is necessary or is just more convenient. This should happen mainly when special bijections are performed.

\begin{nota}\label{indexing}
We recall the notational conventions concerning coordinate projection maps, which are ubiquitous in this paper. Let $A \sub \VF^n \times \RV^m$. For any $n \in \N$, let $I_n = \set{1, \ldots, n}$. First of all, the $\VF$-coordinates and the
$\RV$-coordinates of $A$ are indexed separately. It is cumbersome
to actually distinguish them notationally, so we just assume that
the set of the $\VF$-indices is $I_n$ and the set of the $\RV$-indices is $I_m$. This should never cause confusion in context.

Let $I = I_n \uplus I_m$, $E \sub I$, and $\widetilde{E} = I \mi E$. If $E$ is a
singleton $\set{i}$ then we always write $E$ as $i$ and
$\widetilde{E}$ as $\widetilde{i}$. We write $\pr_E(A)$ for the projection
of $A$ to the coordinates in $E$. For any $\lbar a \in
\pr_{\widetilde{E}} (A)$, the fiber $\{\lbar b : (\lbar b, \lbar a)
\in A \}$ is denoted by $\fib(A, \lbar a)$. Note that we shall often tacitly identify the two subsets $\fib(A, \lbar a)$ and $\fib(A, \lbar a) \times \set{\lbar
a}$. Also, it is often more convenient to use simple descriptions
as subscripts. For example, if $E = \set{1, \ldots, k}$ etc.\ then
we may write $\pr_{\leq k}$ etc. If $E$ contains exactly the
$\VF$-indices (respectively $\RV$-indices) then $\pr_E$ is written
as $\pvf$ (respectively $\prv$). If $E'$ is a subset of the coordinates of $\pr_E (A)$ then the composition $\pr_{E'} \circ \pr_E$ is written as $\pr_{E, E'}$. Naturally
$\pr_{E'} \circ \pvf$ and $\pr_{E'} \circ \prv$ are written as $\pvf_{E'}$ and $\prv_{E'}$, respectively.
\end{nota}

Now, for convenience, we list some of the results in~\cite{Yin:QE:ACVF:min, Yin:special:trans} that shall be used in the sections below. We start with some basic structural properties.

\begin{lem}[{\cite[Lemma~4.11]{Yin:QE:ACVF:min}}]\label{function:dim1:decom:RV}
Let $A, B \sub \VF$ and $f : A \fun B$ a definable surjective function. Then there is a definable function $P : A \fun \RV^m$ such that, for each $\lbar t \in \ran(P)$, $f \rest P^{-1}(\lbar t)$ is either constant or injective.
\end{lem}

\begin{lem}[{\cite[Corollary~3.2]{Yin:special:trans}}]\label{acl:VF:transfer:acl:RV}
For any $\lbar t \in \RV$, any $\lbar t$-definable subset $A \sub \rv^{-1}(\lbar t)$, and any element $x$, if $x \in \acl(\lbar a)$ for every $\lbar a \in A$ then $x \in \acl(\lbar t)$. Similarly, for any $\lbar \gamma \in \Gamma$, any $\lbar \gamma$-definable subset $B \sub \vrv^{-1}(\lbar \gamma)$, and any element $x$, if $x \in \acl(\lbar t)$ for every $\lbar t \in B$ then $x \in \acl(\lbar \gamma)$.
\end{lem}

\begin{lem}[{\cite[Lemma~3.3]{Yin:special:trans}}]\label{dcl:to:ac}
For any $\lbar a, b \in \VF$ and $\lbar t \in \RV$, if $b \in \acl(\lbar a, \lbar t)$ then $b \in \VF(S)(\lbar a)^{\alg}$.
\end{lem}

\begin{lem}[{\cite[Corollary~3.5]{Yin:special:trans}}]\label{function:rv:to:vf:finite:image}
Let $A \sub \RV^m$ and $f: A \fun \VF^n$ a definable function. Then $f(A)$ is finite.
\end{lem}

\begin{lem}[{\cite[Corollary~3.7]{Yin:special:trans}}]\label{alg:ind:imag}
If $a \in \VF$ is such that $a \notin \acl(\0)$, then for any $t \in \RV$ we have $a \notin \acl(t)$. Similarly, if $t \in \RV$ is such that $t \notin \acl(\0)$, then for any $\gamma \in \Gamma$ we have $t \notin \acl(\gamma)$.
\end{lem}

\begin{lem}[{\cite[Lemma~4.13]{Yin:QE:ACVF:min}}]\label{approx:roots}
Let $\gb$ be a ball contained in an $\rv$-ball $t$. Let $G_1(X), \ldots, G_n(X)$ be polynomials with coefficients in $S$. Suppose that $\gb$ does not contain any root of any $G_i(X)$ (hence $\rv$ is constant on every $G_i(\gb)$). If $\gb$ is a closed ball then there is a $d \in t \mi \gb$ such that $\rv(G_i(d)) = \rv(G_i(\gb))$ for every $i$. If $\gb$ is an open ball then there is a $d \in t \mi \gb$ such that $\vv(G_i(d)) = \vv(G_i(\gb))$ for every $i$.
\end{lem}

\begin{lem}[{\cite[Lemma~4.15]{Yin:QE:ACVF:min}}]\label{effectiveness}
Let $\gB$ be an algebraic set of closed balls. Then $\gB$ has centers.
\end{lem}

\begin{lem}[{\cite[Lemma~4.17]{Yin:QE:ACVF:min}}]\label{algebraic:balls:definable:centers}
Suppose that $S$ is $(\VF, \Gamma)$-generated. Let $\gB$ be an algebraic set of balls. Then $\gB$ has centers.
\end{lem}

\begin{cor}\label{function:gamma:to:rv:finite:image}
Let $A \sub \Gamma$ and $f : A \fun \RV$ a definable function. Then $f(A)$ is finite.
\end{cor}
\begin{proof}
By \omin-minimality and compactness, there are only finitely many $t \in \RV$ such that $f^{-1}(t)$ is infinite and every one of them is definable. So, without loss of generality, we may assume that $f$ is finite-to-one. We may also add parameters so that $S$ is $\VF$-generated. By Lemma~\ref{algebraic:balls:definable:centers}, $f(\gamma)$ contains a $\gamma$-definable point for every $\gamma \in A$. It follows from \cmin-minimality that $f(A)$ must be finite.
\end{proof}

Recall from~\cite[Notation~3.16]{Yin:special:trans} that every subset $A \sub \VF^n \times \RV^m$ may be treated as a function from $\pvf(A)$ into the powerset $\mdl P(\RV^m)$ or, sometimes more conveniently, as a function $\VF^n \fun \mdl P(\RV^m)$.

\begin{lem}[{\cite[Lemma~3.18]{Yin:special:trans}}]\label{fun:quo:rep}
Let $G$ be a definable additive subgroup of $\VF$ (hence $G$ is either an open ball around $0$ or a closed ball around $0$). Let $f : \VF \fun \mdl P(\RV^m)$ be a definable function. Then there are $G$-cosets $D_1, \ldots, D_n$ such that $f$ is constant on any  $G$-coset other than $D_1, \ldots, D_n$.
\end{lem}

Recall the notions of $\VF$-dimension and $\RV$-dimension from~\cite[Section~4]{Yin:special:trans} and the terminology that a property holds in a set almost everywhere if it holds outside of a subset of smaller dimension. Whether $\VF$-dimension or $\RV$-dimension is used should be clear in context.

\begin{lem}[{\cite[Corollary~4.7]{Yin:special:trans}}]\label{VF:dim:rv:polydisc}
Let $A$ be a definable subset that contains an $\rv$-polydisc of the form $\{\lbar 0\} \times \rv^{-1}(\lbar t) \times \set{\lbar s}$, where $\lbar 0$ is a tuple of $0 \in \VF$ and $\lbar t \in (\RV^{\times})^k$. Then $\dim_{\VF}(A) \geq k$.
\end{lem}

Let $\ga$ be an open ball and $f : \ga \fun \VF$ an injection. We say that $f$ is \emph{$\rv$-linear} if there is a $t \in \RV$ such that $\rv(f(a) - f(a')) = t \rv(a - a')$ for any $a, a' \in \ga$. Obviously if $f$ is $\rv$-linear then there is only one $t \in \RV$ that satisfies the requirement.

\begin{lem}[{\cite[Lemma~8.5]{Yin:special:trans}}]\label{ball:shrink:RV:cons}
Let $A, B \sub \VF$ be infinite subsets and $f : A \fun B$ a definable bijection. For almost all $a \in A$ there is an $a$-definable $\delta \in \Gamma$ such that $f \rest \go(a, \delta)$ is $\rv$-linear.
\end{lem}

\begin{defn}
A function $f : \VF^n \fun \mdl P(\RV^m)$ is \emph{locally constant at $\lbar a$} if there is an open subset $U_{\lbar a} \sub \VF^n$ containing $\lbar a$ such that $f \rest U_{\lbar a}$ is constant. If $f$ is locally constant at every point in an open subset $A$ then $f$ is locally constant on $A$.
\end{defn}

\begin{lem}[{\cite[Lemma~8.9]{Yin:special:trans}}]\label{fun:almost:loc:con}
Let $f : \VF^n \fun \mdl P(\RV^m)$ be a definable function. Then $f$ is locally constant almost everywhere.
\end{lem}

Recall from~\cite[Section~4]{Yin:special:trans} the definitions of $\VF$-categories and $\RV$-categories without volume forms. The fundamental lifting map $\bb L$ is also defined there (see~\cite[Definition~4.18]{Yin:special:trans}).

\begin{lem}[{\cite[Corollary~4.19]{Yin:special:trans}}]\label{RV:morphism:condition}
Suppose that $F$ is volumetric and there is a definable
function $F^{\uparrow} : \mathbb{L}(U, f) \fun \mathbb{L}(U', f')$
such that the diagram
\[
\bfig
  \square(0,0)/->`->`->`->/<600,400>[\mathbb{L}(U, f)`U_f`\mathbb{L}(U', f')`
   U'_{f'}; \rv`F^{\uparrow}`F_{f, f'}`\rv]
 \square(600,0)/->`->`->`->/<600,400>[U_f`U`U'_{f'}`U';
  \pr_{>k}``F`\pr_{>k}]
\efig
\]
commutes. Then $F$ is a morphism in $\RV[k,\cdot]$.
\end{lem}

\begin{rem}\label{RV:isomorphism:weaker:condition}
In Lemma~\ref{RV:morphism:condition}, if both $F$ and
$F^{\uparrow}$ are bijections then we may drop the assumption that
$F$ is volumetric, since it is guaranteed by the commutative
diagram and Corollary~\ref{VF:dim:rv:polydisc}.
\end{rem}

Recall from~\cite[Definition~5.1]{Yin:special:trans} the notion of a special bijection between two definable subsets. The \emph{$\RV$-hull} of a subset $A$, denoted by $\RVH(A)$, is the union of the $\rv$-polydiscs that have a nonempty intersection with $A$. If $A$ is equal to its $\RV$-hull then $A$ is an \emph{$\RV$-pullback}. An $\RV$-pullback is \emph{degenerate} if it contains a degenerate $\rv$-polydisc and is \emph{strictly degenerate} if it only contains degenerate $\rv$-polydiscs.

\begin{defn}
Let $f : A \fun B$ be a function. We say that $f$ is \emph{contractible} if for every $\rv$-polydisc $\gp \sub \RVH(A)$ the subset $f(\gp \cap A)$ is contained in one $\rv$-polydisc.
\end{defn}

Clearly, if $f : A \fun B$ is a (definable) contractible function then there is a unique (definable) function $f_{\downarrow} : \rv(A) \fun \rv(B)$ such that the diagram commutes:
\[
\bfig
  \square/->`->`->`->/<600,400>[A`B`\rv(A)`\rv(B);  f`\rv`\rv`f_{\downarrow}]
 \efig
\]
We say that $f_{\downarrow}$ is the \emph{contraction} of $f$. Also note that, in this case, if $A$, $B$ are $\RV$-pullbacks and $f$, $f_{\downarrow}$ are bijective then $f$ is a lift of $f_{\downarrow}$ (see~\cite[Definition~7.3]{Yin:special:trans} for the notion of a lift of a function). Equivalently, if $f$ is bijective and both $f$ and $f^{-1}$ are contractible then $f$ is a lift of
$f_{\downarrow}$.

The following technical result plays a crucial role in both~\cite{Yin:special:trans} and this paper (see the lemmas in Section~\ref{section:2unit} as well as Lemma~\ref{bijection:made:contractible} below).

\begin{thm}[{\cite[Theorem~5.5]{Yin:special:trans}}]\label{special:bi:polynomial:constant}
Let $F(\lbar X) = F(X_1, \ldots, X_n)$ be a polynomial with coefficients in $\VF(S)$, $B \sub \VF^n$ a definable subset, $\tau : B \fun A$ a special bijection, and $f = F \circ \tau^{-1}$. Then there is a special bijection $T$ on $A$ such that $T(A)$ is an $\RV$-pullback and $f \circ T^{-1}$ is contractible.
\end{thm}

\begin{cor}[{\cite[Corollary~5.6]{Yin:special:trans}}]\label{all:subsets:rvproduct}
Every definable subset $A \sub \VF^n \times \RV^m$ is a definable
deformed $\RV$-pullback.
\end{cor}

Here is Step~1 for categories without volume forms:

\begin{thm}[{\cite[Corollary~5.7]{Yin:special:trans}}]\label{L:surjective}
The map $\mathbb{L}: \ob \RV[k, \cdot] \fun \ob \VF[k, \cdot]$ is
surjective on the isomorphism classes of $\VF[k, \cdot]$. The map $\mathbb{L}: \ob \RV[k] \fun \ob \VF[k]$ is surjective on the isomorphism classes of $\VF[k]$.
\end{thm}

Recall the notion of a $\lbar \gamma$-polynomial (\cite[Definition~7.1]{Yin:special:trans}). There is a version of Hensel's lemma for such polynomials:

\begin{lem}[{\cite[Lemma~7.2]{Yin:special:trans}}]\label{hensel:lemma}
Let $F_1(\lbar X), \ldots, F_n(\lbar X)$ be $\lbar \gamma$-polynomials with residue values $\alpha_1, \ldots, \alpha_n$, where $\lbar \gamma = (\gamma_1, \ldots, \gamma_n) \in \Gamma$. For every simple common residue root $\lbar t = (t_1,
\ldots, t_n) \in \RV$ of $F_1(\lbar X), \ldots, F_n(\lbar X)$
there is a unique $\lbar a \in \rv^{-1}(\lbar t)$ such that
$F_i(\lbar a) = 0$ for every $i$.
\end{lem}

To apply this generalized Hensel's lemma we usually need the following:

\begin{lem}[{\cite[Lemma~7.4]{Yin:special:trans}}]\label{exists:gamma:polynomial}
Suppose that $S$ is $(\VF, \Gamma)$-generated. Let $\lbar t = (\lbar t_n, t_{n}) \in \RV$ with $t_n \in \acl(\lbar t_n)$ and $\vrv (\lbar t) = \lbar \gamma \in \Gamma$. Then there is a $\lbar \gamma$-polynomial $F(\lbar X)$ with coefficients in $\VF(S)$ such that $\lbar t$ is a residue root of $F(\lbar X)$ but is not a residue root of $\partial F(\lbar X) / \partial X_n$.
\end{lem}

Here is Step~2 for categories without volume forms:

\begin{thm}[{\cite[Corollary~7.7]{Yin:special:trans}}]\label{L:semigroup:hom}
Suppose that the substructure $S$ is $(\VF, \Gamma)$-generated.
The map $\mathbb{L}$ induces surjective homomorphisms between various
Grothendieck semigroups, for example:
\[
\gsk \RV[k, \cdot] \fun \gsk \VF[k, \cdot],\quad
\gsk \RV[k] \fun \gsk \VF[k].
\]
\end{thm}

There are two different and yet compatible approaches to defining the Jacobian of a morphism in the $\VF$-categories (\cite[Section~9]{Yin:special:trans}): one analytic, the other algebraic, each has its own advantages. The algebraic approach may also be used to define the Jacobian of a morphism in the $\RV$-categories.

\begin{thm}[{\cite[Corollary~9.9]{Yin:special:trans}}]\label{diff:almost:every}
Let $f : \VF^n \fun \VF^m$ be a definable function. Then $f$ is continuously partially differentiable almost everywhere.
\end{thm}

\begin{lem}[{\cite[Lemma~9.11]{Yin:special:trans}}]\label{special:tran:vol:pre}
For any special bijection $T : A \fun A^{\sharp}$, the Jacobians of $T$ and $T^{-1}$ are equal to $1$ almost everywhere. If $A$ is a nondegenerate $\RV$-pullback then they are equal to $1$ everywhere.
\end{lem}

\begin{lem}[{\cite[Lemma~9.12]{Yin:special:trans}}]\label{jcb:chain}
Let $f : A \fun B$ and $g : B \fun C$ be definable functions. Then for any $\lbar x \in A$,
\[
\jcb_{\VF} ( g \circ f)(\lbar x) = \jcb_{\VF} g(f(\lbar x)) \cdot \jcb_{\VF} f(\lbar x),
\]
if both sides are defined.
\end{lem}

Recall from~\cite[Definition~9.13]{Yin:special:trans} the notion of an essential isomorphism between two objects in $\RV[k]$. Note that conceptually this is quite different from the notion of an essential bijection between two objects in $\VF[k]$ (see~\cite[Definition~10.1]{Yin:special:trans}).

\begin{lem}[{\cite[Lemma~9.15]{Yin:special:trans}}]\label{RV:iso:class:lifted:jcb}
Let $F : (U, f)\fun (V, g)$ be an essential $\RV[k]$-isomorphism and $F^{\uparrow} : \bb L(U, f) \fun \bb L (V, g)$ a lift of $F$. Then for all $\lbar u \in U$ outside a definable subset of $U$ of dimension $< k$ and almost all $(\lbar a, \lbar u) \in \rv^{-1}(f(\lbar u), \lbar u)$,
\[
\rv(\jcb_{\VF} F^{\uparrow}(\lbar a, \lbar u)) = \jcb_{\RV} F(f(\lbar u), \lbar u).
\]
Also, for almost all $(\lbar a, \lbar u) \in \bb L(U, f)$,
\[
\vv(\jcb_{\VF} F^{\uparrow}(\lbar a, \lbar u)) = \jcb_{\Gamma} F(f(\lbar u), \lbar u).
\]
\end{lem}

We now consider the $\VF$-categories and the $\RV$-categories with volume forms (see~\cite[Section~10]{Yin:special:trans}). For any $(\mathbf{U}, \omega) \in \mRV[k]$, let $\bb L \omega$ be the function on $\bb L \mathbf{U}$ naturally induced by $\omega$. The \emph{lift} of $(\mathbf{U}, \omega)$ is the object $\bb L(\mathbf{U}, \omega) = (\bb L \mathbf{U}, \bb L \omega) \in \mVF[k]$.


\begin{cor}\label{lift:iso:mrv:iso}
Let $(\mathbf{U}, \omega), (\mathbf{V}, \pi) \in \mRV[k]$ and $F^{\uparrow} : \bb L \mathbf{U} \fun \bb L \mathbf{V}$ a lift of an essential $\RV[k]$-isomorphism $F: \mathbf{U} \fun \mathbf{V}$. If $F^{\uparrow}$ is measure-preserving with respect to $\bb L \omega$, $\bb L \pi$ then $F$ is a $\mRV[k]$-isomorphism between $(\mathbf{U}, \omega)$ and $(\mathbf{V}, \pi)$.
\end{cor}


Here are Step~1 and Step~2 for the categories with volume forms:

\begin{thm}[{\cite[Theorem~10.4]{Yin:special:trans}}]\label{L:measure:surjective}
Every object $(A, \omega)$ in $\mVF[k]$ is isomorphic to another object $\bb L(\mathbf{U}, \pi)$ in $\mVF[k]$, where $(\mathbf{U}, \pi) \in \mRV[k]$; similarly for other pairs of corresponding categories.
\end{thm}

\begin{thm}[{\cite[Corollary~10.6]{Yin:special:trans}}]\label{L:measure:semigroup:hom}
The map $\mathbb{L}$ induces surjective homomorphisms between the various
Grothendieck semigroups associated with the categories with volume forms, for example:
\[
\gsk \mRV[k] \fun \gsk \mVF[k], \quad \gsk \mgRV[k] \fun \gsk \mgVF[k].
\]
\end{thm}

\begin{nota}\label{nota:iso:classes}
For any subset $E \sub \mathds{N}$ with $\abs{E} = k$, we write $[U]_{E}$ and $[U, \omega]_{E}$ for the isomorphism classes $[(U, \pr_{E})] \in \gsk \RV[k,\cdot]$ and $[(U, \pr_{E}, \omega)] \in \gsk \mRV[k]$, respectively. If $\omega$ is the constant form $1$ then $[U, \omega]_{E}$ is simply written as $[U]_{E}$ as well. If $E = I_k$ etc.\ then we may write $[U]_{\leq k}$, $[U, \omega]_{\leq k}$, etc. If in context it is clear that certain function $f: U \fun \RV^k$ is being used then we just write $[U]_{k}$ and $[U, \omega]_{k}$ for the isomorphism classes $[(U, f)] \in \gsk \RV[k,\cdot]$ and $[(U, f, \omega)] \in \gsk \mRV[k]$. For example, any two singletons $(\{t\}, f), (\{s\}, g) \in \RV[k,\cdot]$ are isomorphic, so we may write $[1]_k$ for their isomorphism class. Similarly, if $t, s \in \K^{\times}$ then $(\{t\}, \id, 1), (\{s\}, \id, 1) \in \mRV[1]$ are isomorphic and we write $[1]_1$ for this isomorphism class.
\end{nota}

\section{Parametric balls and atomic subsets}\label{section:atoms}

In this section we shall work in the expansion $\gC^{\bullet}$. The main goal is to establish the open-to-open property for definable bijections between two subsets of $\VF$ (see Proposition~\ref{partition:open:to:open} below).

\begin{defn}
Let $Q$ be a set of parameters. We say that a (not necessarily definable) subset $A \sub \VF^n \times \RV^m$ \emph{generates a complete $Q$-type} if, for every $Q$-definable subset $B$, either $A \sub B$ or $A \cap B = \0$. If $A$ is $Q$-definable and generates a complete $Q$-type then we say that it is \emph{$Q$-atomic} or \emph{atomic over $Q$}.
\end{defn}

For example, if $t \in \RV$ is not algebraic then $\rv^{-1}(t)$ is $t$-atomic. More generally, we have

\begin{lem}\label{atomic:over:its:name}
Let $\gB$ be a definable set of balls and $\phi$ a formula such that, for all $t_1 \neq t_2 \in \gB$, $\phi(t_1)$ and $\phi(t_2)$ define two disjoint balls $\gb_{t_1}$ and $\gb_{t_2}$. For each $t \in \gB$, if $\dot \gb_t$ is not algebraic then $\gb_t$ is $t$-atomic.
\end{lem}
\begin{proof}
Suppose for contradiction that there is a non-algebraic $\dot \gb_s$ and a formula $\psi$ such that $\psi(s)$ defines a proper subset of $\gb_s$. For each $t \in \gB$, let $A_t$ be the set defined by $\psi(t)$ if it is a proper subset of $\gb_t$ and $A_t = \0$ otherwise. Set $A = \bigcup_{t \in \gB} A_t$, which is definable. By \cmin-minimality, $A$ is a boolean
combination of some balls $\gd_1, \ldots, \gd_n$. Since the balls
$\gb_t$ are pairwise disjoint, there are only finitely many balls
$\gb_t$ that contain some $\gd_i$. Note that the set of these finitely many balls is definable, which does not contain
$\gb_s$ since $\dot \gb_s$ is not algebraic. On the other
hand, since $\gb_s \cap A \neq \0$, we must have $\gb_s \sub
A$. This is a contradiction because the balls $\gb_t$ being
pairwise disjoint implies that $\gb_s \cap A$ is a proper subset
of $\gb_s$.
\end{proof}

\begin{lem}\label{atomic:base:change}
Let $A \sub \VF^n \times \RV^m$ be atomic. Then $A$ is $\lbar \gamma$-atomic for all $\lbar \gamma \in \Gamma$.
\end{lem}
\begin{proof}
By induction this is immediately reduced to the case that the length of $\lbar \gamma$ is $1$. Suppose for contradiction that there is a formula $\psi(\gamma)$ that defines a proper subset of $A$. Then the subset
\[
\Delta = \set{\gamma \in \Gamma : \psi(\gamma) \text{ defines a proper subset of } A}
\]
is nonempty and is definable. By $o$-minimality, some $\alpha \in \Delta$ is definable, contradicting the assumption that $A$ is atomic.
\end{proof}

\begin{lem}\label{subball:atomic}
Let $A \sub \VF$ generate a complete type and $\gb$ an open (or closed) ball contained in $A$. Then $\gb$ is $\dot \gb$-atomic.
\end{lem}
\begin{proof}
We assume that $\gb$ is an open ball, since the argument for closed
balls is identical. Suppose for contradiction that there is a formula $\psi$ such that $\psi(\dot \gb)$ defines a proper subset of $\gb$. By Lemma~\ref{atomic:over:its:name}, there is a finite definable subset $\gB$ of balls with $\dot \gb \in \gB$. For each $\dot \gd \in \gB$, let $B_{\dot \gd}$ be the subset defined by $\psi(\dot \gd)$ if it is a proper subset of $\gd$ and $B_{\dot \gd} = \0$ otherwise. Then $B = \bigcup_{\dot \gd \in \gB} B_{\dot \gd}$ is a definable subset that neither contains $A$ nor is disjoint from $A$, contradiction.
\end{proof}

\begin{defn}
Let $\gb_1$ and $\gb_2$ be two punctured balls. We say that they
are of the same \emph{primitive type} if
\begin{enumerate}
 \item $\rad(\gb_1) = \rad(\gb_2)$ and $\vcr(\gb_1) =
\vcr(\gb_2)$,
 \item they are both open balls or both closed balls or both thin annuli.
\end{enumerate}
\end{defn}

\begin{lem}\label{atomic:subset:types}
Every atomic $A \sub \VF$ is the union of disjoint punctured balls $\gb_1, \ldots, \gb_n$ of the same primitive type.
\end{lem}
\begin{proof}
By \cmin-minimality, $A$ is a union of disjoint punctured balls $\gb_1,
\ldots, \gb_n$. First of all,
since $A$ is atomic, both $\vcr$ and $\rad$ must be constant on
$\set{\gb_1, \ldots, \gb_n}$, because otherwise there would be a definable proper subset of $A$ according to $\min \set{\vcr(\gb_1), \ldots, \vcr(\gb_n)}$ or $\min \set{\rad(\gb_1), \ldots, \rad(\gb_n)}$. Similarly either $\gb_1, \ldots, \gb_n$ are all closed balls or are all open balls. Also, since the subset of
$A$ that contains exactly every unpunctured ball $\gb_i$ is
definable, we have that either $\gb_1, \ldots, \gb_n$ are all
punctured or are all unpunctured.

So it is enough to show that if $\gb_i$ is punctured then it must
be a thin annulus. By atomicity again, if $\gb_1, \ldots, \gb_n$
are punctured then each $\gb_i$ must contain the same number of
holes. If $\gb_i$ has a hole $\gh$ with $\rad(\gh) < \rad(\gb_i)$
then $\gb_i \mi \gh^*$ is nonempty, where $\gh^*$ is the closed
ball that has radius $(\rad(\gb_i) + \rad(\gh))/2$ and contains
$\gh$. The collection of all such holes $\gh_1, \ldots, \gh_m$ is definable and hence, if it is not empty, then there
would be a proper subset of $A$ that is defined by replacing each $\gh_i$ with $\gh_i^*$. So each $\gb_i$ is a closed ball and each hole in each
$\gb_i$ is a maximal open ball in $\gb_i$.

Suppose for contradiction that $\gb_1$ contains more than one holes $\gh_1,
\ldots, \gh_m$. Recall from Remark~\ref{rem:K:aff} that each closed ball carries a $\K$-affine structure. This means that the subsets
\[
1 \cdot(\gh_2 - \gh_1), \ldots, (m+1)\cdot(\gh_2 - \gh_1)
\]
are distinct $\rv$-balls. Therefore, for some $1 \leq k \leq m+1$ we have
that $\gh_1 + k \cdot (\gh_2 - \gh_1)$ is a maximal open ball in
$\gb_1$ and is disjoint from $\bigcup_i \gh_i$. This means that
there is a finite definable set of maximal open balls in
$\gb_1, \ldots, \gb_n$ that strictly contains the set of holes in
$\gb_1, \ldots, \gb_n$. This readily implies that $A$ has a
nonempty proper definable subset, contradiction.
\end{proof}

Note that, in the above lemma, if $A$ is definable in $\gC$ then it cannot be a
disjoint union of closed balls of radius $< \infty$, because in
that case, by Lemma~\ref{effectiveness}, the closed balls would
have definable centers. Now, by the above lemma, the \emph{radius} and the \emph{valuative center} of $A$ are well-defined quantities: they are respectively the radius and the valuative center of the balls $\gb_1, \ldots,
\gb_n$ in the above lemma. These are also denoted by $\rad(A)$ and
$\vcr(A)$. The balls $\gb_1, \ldots, \gb_n$ are called the
\emph{primitive components} of $A$.

\begin{cor}
If $A \sub \VF$ is atomic and $\gb \sub A$ is an open (closed) ball then every $a \in A$ is contained in an open (closed) ball $\gd_a \sub A$ with $\rad(\gd_a) = \rad(\gb)$.
\end{cor}

\begin{lem}\label{hae:component:11}
Let $A \sub \VF$ be atomic with only one primitive component. Let $f : A \fun \VF$ be a definable function. Then $f(A)$ also has only one primitive component.
\end{lem}
\begin{proof}
Let $\gb_1, \ldots, \gb_n$ be the primitive components of
$f(A)$ given by Lemma~\ref{atomic:subset:types}. Suppose for
contradiction that $n > 1$. Then, by $C$-minimality, there is exactly one of these
components, say $\gb_1$, such that
$f^{-1}(\gb_1)$ is a punctured ball of the form $A \mi \bigcup_j
\gh_j$ for some holes $\gh_j$. Consequently, the ball $\gb_1$ and hence $f^{-1}(\gb_1)$ are definable, contradicting the assumption that $A$ is
atomic.
\end{proof}

\begin{lem}\label{atom:open:RV:constant}
Let $\ga$ be an atomic open ball and $f : \ga \fun \mdl P (\RV^m)$ a definable function. Then $f$ is constant.
\end{lem}
\begin{proof}
By Lemma~\ref{atomic:base:change}, $\ga$ remains atomic over any $\gamma > \rad(\ga)$. Then, by Lemma~\ref{fun:quo:rep}, $f$ is constant on every open subball of $\ga$ of any radius $> \rad(\ga)$. Hence $f$ must be constant.
\end{proof}

Combining Lemma~\ref{subball:atomic} and Lemma~\ref{atom:open:RV:constant}, we get

\begin{cor}
Let $A \sub \VF$ be atomic and $f : A \fun \mdl P (\RV^m)$ a definable function. Then $f$ is constant on every open subball of $A$.
\end{cor}

The following lemma says that, in a geometrical sense, an atomic open ball is really different from an atomic closed ball or an atomic thin annulus.

\begin{lem}\label{atomic:balls:no:alg:relation}
Let $\go$ be an open ball and $\gl$ a closed ball of radius $< \infty$ or a thin annulus. Suppose that both $\go$ and $\gl$ are atomic. If $A \sub \go \times \gl$ is definable then the projection $\pr_1 \rest A$ cannot be finite-to-one.
\end{lem}
\begin{proof}
We assume that $\gl$ is a closed ball, since the proof for thin
annuli is identical. Suppose for contradiction that there is definable $A \sub \go \times \gl$ such that $\pr_1 \rest A$ is finite-to-one. Since
$\go$ and $\gl$ are atomic, we must have $\pr_1 (A) = \go$ and
$\pr_2 (A) = \gl$. Let $\gM$ be the set of maximal open subballs of
$\gl$, which is definable. For any $\dot \gx \in \gM$, let $A_{\gx} = \bigcup_{b \in \gx}\fib(A, b) \sub \go$. By \cmin-minimality each $A_{\gx}$ is a boolean combination of balls. In fact, for any $\dot \gx, \dot \gy \in \gM$, $A_{\gx}$ and $A_{\gy}$ must have the same
number of boolean components, because otherwise there would be a definable proper subset of $\gl$. Let this number be $k$.

Let $\dot \gx \in \gM$ and suppose that $\gB = \set{\gb_1, \ldots,
\gb_k}$ is the set of the boolean components of $A_{\gx}$. Set $\lambda_{\gx} = \min \set{\rad(\gb_1), \ldots, \rad(\gb_k)}$. For any $\gb_i, \gb_j \in \gB$, let
\[
\begin{cases}
  \rho(\gb_i, \gb_j) = \min \set{\rad(\gb_i), \rad(\gb_j)}, & \text{if } \gb_i \cap \gb_j \neq \0,\\
  \rho(\gb_i, \gb_j) = \vv(\gb_i - \gb_j), & \text{otherwise}.
\end{cases}
\]
Let
\[
\rho_{\gx} = \min \set{\rho(\gb_i, \gb_j) : \gb_i, \gb_j \in \gB}.
\]
Note that the subsets $\Lambda = \set{\lambda_{\gx} : \dot \gx \in \gM} \sub \Gamma$ and $\Delta = \set{\rho_{\gx} : \dot \gx \in \gM} \sub \Gamma$ are both definable. Since $\gl$ is atomic, we must have that both
$\Lambda$ and $\Delta$ are singletons, say $\Lambda =
\set{\lambda}$ and $\Delta = \set{\rho}$.

We claim that $\lambda > \rad(\go)$. To see this, suppose for contradiction
$\lambda_{\gx} = \rad(\go)$ for every $\dot \gx \in \gM$. This means
that every $A_{\gx}$ has $\go$ as a positive boolean component. Since $\go$ is open, we have that for any $\dot \gx_1, \ldots, \dot \gx_n \in \gM$ the intersection $\bigcap_{i \leq n} A_{\gx_i}$ is nonempty and hence there are $b_i \in \gx_i$ such that $(a, b_i) \in A$. Therefore, by compactness,
there is an $a \in \go$ such that $\fib(A, a)$ is infinite, contradicting the assumption that $\pr_1 \rest A$ is finite-to-one.

Now, fix an $\dot \gx \in \gM$. Since $\go$ is open and $\pr_1 (A) = \go$, there is a proper open subball $\gz$ of $\go$ that properly contains
$A_{\gx}$. Let $B_{\gz} = \bigcup_{a \in \gz}\fib(A, a) \sub \gl$. Since $B_{\gz}$ properly contains the maximal open subball $\gx$ of
$\gl$, either $\gx$ is a boolean component of
$B_{\gz}$ that is disjoint from any other boolean component of
$B_{\gz}$ or $\gl$ is a positive boolean component of $B_{\gz}$.
However, the former is impossible, because in that case $B_{\gz}$
could only have finitely many maximal open subballs of $\gl$ as
its positive boolean components and consequently, since $\Lambda =
\set{\lambda}$ and $\lambda > \rad(\go)$, $\gz$ could not be an open ball that properly contains $A_{\gx}$, contradiction. So we must have that $\gl$ is a positive boolean component of $B_{\gz}$. This means that $B_{\gz}$ can only have finitely many maximal open subballs of
$\gl$ as its negative boolean components, say $\gx_1, \ldots,
\gx_n$. Again, since $\Lambda = \set{\lambda}$ and $\lambda >
\rad(\go)$, $\bigcup_{i \leq n} A_{\gx_i}$ must be a proper subset
of $\go \mi \gz$ and hence there is a $\dot \gy \in \gM$ such that $\gy
\sub B_{\gz}$ and $A_{\gy}$ has a boolean component contained in
$\gz$ and another boolean component disjoint from $\gz$. This
implies that $\rho_{\gy} \leq \rad (\gz)$. On the other hand,
since $A_{\gx} \sub \gz$ and $\gz$ is an open ball, we have $\rho_{\gx} > \rad (\gz)$. This is a contradiction since $\Delta$ is a singleton.
\end{proof}

\begin{cor}\label{atomic:no:function}
There are no definable finite-to-one functions from an atomic closed ball or an atomic thin annulus to an atomic open ball. There are no definable functions at all from an atomic open ball to an atomic closed ball of radius $< \infty$ or an atomic thin annulus.
\end{cor}

Combining Lemma~\ref{hae:component:11} and Corollary~\ref{atomic:no:function}, we get:

\begin{lem}\label{atomic:open:ball:open}
Let $A \sub \VF$ be an atomic open ball and $f : A \fun \VF$ a definable function. Then either $f$ is constant or $f(A)$ is also an atomic open ball.
\end{lem}

\begin{lem}\label{atomic:open:no:atomic:closed}
Let $\gb$ be a $\dot \gb$-atomic open ball. Let $A \sub \VF$ be an infinite $\dot \gb$-atomic subset with only one primitive component. Suppose that the only imaginary parameter needed to define $A$ is $\dot \gb$. Then $A$ is not a closed ball.
\end{lem}
\begin{proof}
Suppose for contradiction that $A$ is a closed ball. There is a quantifier-free $\lan{RV}$-formula $\psi(X, Y)$ such that $\psi(X, b)$ defines $A$ for every $b \in \gb$. Let $F_i(X, Y)$ be the occurring polynomials of $\psi(X, Y)$. Since $\gb$ and $A$ are $\dot \gb$-atomic, we may assume that every $F_i(X, b)$ and $F_i(a, Y)$ are nonzero for all $b \in \gb$ and $a \in A$. Also note that for all $a \in A$, $b \in \gb$, and $F_i(X, Y)$ we have $F_i(a, b) \neq 0$, for otherwise there would be a $\dot \gb$-definable $C \sub \gb \times A$ with $\pr_1 \rest C$ finite-to-one, contradicting Lemma~\ref{atomic:balls:no:alg:relation}. Now fix a $b \in \gb$. Since all the roots of all $F_i(X, b)$ lie outside of $A$ and $A$ is a closed ball, by Lemma~\ref{approx:roots}, there is a $d \notin A$ such that $\rv(F_i(d, b)) = \rv(F_i(A, b))$ for all $F_i(X, b)$. So $d$ also satisfies $\psi(X, b)$, contradiction.
\end{proof}

\begin{lem}\label{complete:type:open:ball}
Let $A \sub \VF$ generate a complete type. Let $f : \VF \fun \VF$ be a definable function such that $f \rest A$ is injective. Then for every open ball $\gb \sub A$ the image $f(\gb)$ is also an open ball.
\end{lem}
\begin{proof}
For any open ball $\gb \sub A$, by Lemma~\ref{subball:atomic}, $\gb$ is $\dot \gb$-atomic. By Lemma~\ref{atomic:open:ball:open}, $f(\gb)$ is an open ball.
\end{proof}

Now we introduce a (coarsened) variation of $\rv$-linearity. Let $\ga$ be an open ball and $f : \ga \fun \VF$ an injection. We say that $f$ is \emph{$\Gamma$-linear} if there is a $\gamma \in \Gamma$ such that $\vv(f(a) - f(a')) = \gamma + \vv(a - a')$ for any $a, a' \in \ga$. Obviously if $f$ is $\Gamma$-linear then there is only one $\gamma \in \Gamma$ that satisfies the requirement.

\begin{lem}\label{atomic:ball:shrink:RV:cons}
Let $f : \ga \fun \gb$ be a definable bijection between two atomic open balls. Then $f$ is $\Gamma$-linear with respect to $\rad(\gb) - \rad(\ga)$.
\end{lem}
\begin{proof}
Let $U = \{t \in \RV : \vrv(t) > \rad(\ga)\}$ and $V = \{t \in \RV : \vrv(t) > \rad(\gb)\}$. By Lemma~\ref{complete:type:open:ball}, $f$ has the open-to-open property. So, for any $a \in \ga$, the $a$-definable function $f_a : (\ga - a) \fun (\gb - f(a))$ given by $a' -a \efun f(a') - f(a)$ naturally induces a bijection $h_a : U \fun V$. Note that, for any $t, t' \in U$, $\vrv(t) = \vrv(t')$ if and only if $\vrv(h_a(t)) = \vrv(h_a(t'))$. Since $\ga$ is atomic, by Lemma~\ref{atom:open:RV:constant}, all these bijections $h_a$ are actually the same and hence may be denoted by $h$. Let $\hat h : U \fun \RV$ be the function given by $t \efun h(t)/ t$. Let $a, b \in \ga$ and $u \in U$ such that $\rv(b - a) = u$. Since $\rv^{-1}(u) = \rv^{-1}(2u) + a - b$, we have
\[
f_a(\rv^{-1}(2u)) - f_a(\rv^{-1}(u)) = f_b(\rv^{-1}(u)).
\]
So $h(2u) = 2h(u)$. In fact, by the same argument, we see that $h(ku) = kh(u)$ and hence $\hat h(ku) = \hat h(u)$ for all nonzero integer $k$. It follows from strong minimality that, for each $\gamma > \rad(\ga)$, the restriction $\hat h \rest (\vrv^{-1}(\gamma) \cap U)$ is constant and hence $\hat h$ may be regarded as a function from the interval $(\rad(\ga), \infty)$ into $\RV$. By Corollary~\ref{function:gamma:to:rv:finite:image}, $\hat h(U)$ is finite. Since if $\vrv(t) \leq \vrv(t')$ then $\vrv(\hat h(t)) \leq \vrv(\hat h(t'))$, we must have $\vrv(\hat h(U)) = \{\rad(\gb) - \rad(\ga)\}$.
\end{proof}

\begin{cor}\label{rv:linear:rad}
Let $f : \ga \fun \gb$ be a definable bijection between two open balls. For each $a \in \ga$ let $\delta_a \in \Gamma \cup \{\infty\}$ be the least such that $f \rest \go(a, \delta_a)$ is $\Gamma$-linear. Then for no $\beta > \rad(\ga)$ do we have $\beta < \delta_a$ for every $a \in \ga$.
\end{cor}
\begin{proof}
Suppose for contradiction that there is a $\beta > \rad(\ga)$ such that $\beta < \delta_a$ for every $a \in \ga$. By Lemma~\ref{atomic:over:its:name} there is an atomic open ball $\gd \sub \ga$ with $\rad(\gd) = \beta$. By Lemma~\ref{atomic:ball:shrink:RV:cons}, $f \rest \gd$ is $\Gamma$-linear, contradiction.
\end{proof}

For our purposes in this paper, $\Gamma$-linearity is enough. On the other hand, some of the results still hold if we replace $\Gamma$-linearity with $\rv$-linearity, because we have the following strengthening of Lemma~\ref{atomic:ball:shrink:RV:cons}:

\begin{lem}\label{atomic:ball:rv:linear}
Let $f : \ga \fun \gb$ be a definable bijection between two atomic open balls. Then $f$ is $\rv$-linear.
\end{lem}
\begin{proof}
Let $f_a, U, V, \hat h$ be as in the proof of Lemma~\ref{atomic:ball:shrink:RV:cons} and $\gamma = \rad(\gb) - \rad(\ga)$. Fix an $e \in \ga$ and we shall work with $\ga - e = \rv^{-1}(U)$. Let $\delta \in \Gamma$ be the least such that $\hat h \rest (\delta, \infty)$ is constant. We may assume $\delta > \rad(\ga)$ (otherwise we are done). Let $\hat h (\delta, \infty) = r$ and we claim that $\hat h(\delta) = r$. To see this, we first choose a $c \in \rv^{-1}(\hat h(\delta))$. Let $t \in U$ such that $\vrv(t) = \delta$ and $\rv^{-1}(t)$ is atomic over $\dcl(e, c, t)$, which is possible by Lemma~\ref{atomic:over:its:name}. Hence, by Lemma~\ref{atom:open:RV:constant}, the set $\{\rv(f_e(a) - ca) : a \in \rv^{-1}(t)\}$ is a singleton, say $\{s\}$. Note that $\vrv(s) > \delta + \gamma$. Consider any $b \in \rv^{-1}(t)$ and any $b' \in \go(b, \delta)$ with $\vv(b' - b) < \vrv(s) - \gamma$. Since
\[
\delta + \gamma < \vv(c(b' - b)) < \vrv(s) \quad \text{and} \quad \vv((f_e(b') - c b') - (f_e(b) - c b)) > \vrv(s),
\]
we must have $\rv(f_e(b') - f_e(b)) = \rv(c b' - c b)$ and hence $\rv(c) = \hat h(\rv(b' - b)) = r$.

Now we need to show that $\hat h$ is constant. The argument is very similar to the one above. Suppose for contradiction that $\hat h$ is not constant. Then there are definable $r' \neq r \in \hat h(U)$ and $\delta' < \delta \in \Gamma$ such that $\hat h[\delta, \infty) = r$ and $\hat h(\delta', \delta) = r'$. Let
\[
A = \{a \in \ga - e : \delta' < \vv(a) < \delta\}.
\]
Choose a $c \in \rv^{-1}(r')$ and let $g : A \fun \RV$ be the function given by $a \efun \rv(f_e(a) - ca)$. For notational simplicity, we may assume $e, c \in S$. By Lemma~\ref{alg:ind:imag} and Lemma~\ref{atomic:over:its:name}, if $\alpha < \delta$ is sufficiently close to $\delta$ (that is, larger than every definable element in $(\delta', \delta)$) and $\vrv(t) = \alpha$ then $\rv^{-1}(t)$ is $t$-atomic and hence, by Lemma~\ref{atom:open:RV:constant}, $g \rest \rv^{-1}(t)$ is constant. For each $\delta' < \alpha < \delta$ let $\zeta(\alpha + \gamma)$ be an $\alpha$-definable element in $\{\vrv(g(a)) : \vv(a) = \alpha\}$. Note that $\zeta(\alpha + \gamma) > \alpha + \gamma$. By quantifier elimination (in $\lan{v}$), the definable function $\zeta : (\delta' + \gamma, \delta + \gamma) \fun \Gamma$ is piecewise linear. This means that there is a definable $\alpha \in (\delta', \delta)$ such that $\zeta \rest (\alpha + \gamma, \delta + \gamma)$ is given by a linear equation and hence, if $\beta + \gamma \in (\alpha + \gamma, \delta + \gamma)$ is sufficiently close to $\delta + \gamma$ then $\zeta(\beta + \gamma) > \delta + \gamma$. Let $\beta < \delta$ be sufficiently close to $\delta$ and $t \in \vrv^{-1}(\beta)$ such that $\zeta(\beta + \gamma) = \vrv(g(\rv^{-1}(t)))$. For any $b \in \rv^{-1}(t)$ and any $b' \in \go(b, \delta)$, \[
\text{if} \quad \delta < \vv(b' - b) < \zeta(\beta + \gamma) - \gamma \quad \text{then} \quad \delta + \gamma < \vv(c(b' - b)) < \zeta(\beta + \gamma).
\]
In this case, since
\[
\vv((f_e(b') - c b') - (f_e(b) - c b)) > \zeta(\beta + \gamma),
\]
we have $\rv(f_e(b') - f_e(b)) = \rv(c b' - c b)$ and hence $\rv(c) = \hat h(\rv(b' - b)) = r$, contradiction.
\end{proof}

The above lemma is~\cite[Lemma~5.5]{hrushovski:kazhdan:integration:vf}, which plays a very important role in ~\cite{hrushovski:kazhdan:integration:vf} but is not needed here.

\begin{prop}\label{partition:open:to:open}
Let $f : A \fun B$ be a definable bijection between two subsets of $\VF$. Then there is a partition of $A$ into definable subsets $A_1, \ldots, A_n$ such that, for all $\ga \sub A_i$, $\ga$ is an open ball if and only if $f(\ga)$ is an open ball (we say that each $f \rest A_i$ has the \emph{open-to-open property}).
\end{prop}
\begin{proof}
For every $a \in A$ let $D_a \sub A$ be the intersection of all
definable subsets of $A$ that contains $a$. So $D_a$ generates a complete type. By Lemma~\ref{complete:type:open:ball}, for every
open ball $\ga \sub D_a$, the image $f(\ga)$ is an open ball. By compactness, this property holds in a definable subset $A_a \sub A$ that contains $a$. By compactness
again, there are definable subsets $A_1, \ldots, A_m \sub A$ with
$\bigcup A_i = A$ such that each $A_i$ has this property. Similarly there are definable subsets $B_1, \ldots, B_l \sub B$ with $\bigcup B_i = B$ such that each $B_i$ has this property (with respect to $f^{-1}$). The partition of $A$ generated by $A_1, \ldots, A_m$, $f^{-1}(B_1), \ldots, f^{-1}(B_l)$ is as desired.
\end{proof}


It is more convenient to deal with the open-to-open property in the above proposition in more general terms:

\begin{defn}\label{def:open:open}
Let $A \sub \VF^{n_1} \times \RV^{m_1}$, $B \sub \VF^{n_2} \times
\RV^{m_2}$, and $f : A \fun B$ a bijection. Let $i \in I_{n_1}$ and $j \in I_{n_2}$. For any $\lbar a
\in \pr_{\widetilde{i}}(A)$ and any $\lbar b \in \pr_{\widetilde{j}}(B)$,
let
\[
f_{\lbar a, \lbar b} = f \rest (\fib(A, \lbar a) \cap f^{-1}(\fib(B, \lbar b))).
\]
We say that $f$ has the \emph{$(i,j)$-open-to-open} property if, for every $\lbar a \in \pr_{\widetilde{i}}(A)$ and every $\lbar b \in \pr_{\widetilde{j}}(B)$, $f_{\lbar a, \lbar b}$ has the open-to-open property. If $f$ has the
$(i,j)$-open-to-open property for every $(i,j) \in I_{n_1} \times I_{n_2}$ then $f$ has the open-to-open property.
\end{defn}

With this understanding, Proposition~\ref{partition:open:to:open}
may be easily generalized as follows:

\begin{prop}\label{open:to:open:parameter}
Let $A$, $B$, and $f$ be as above. Then there is a partition of $A$ into definable subsets $A_1, \ldots, A_n$ such that every $f \rest A_i$ has the open-to-open property.
\end{prop}
\begin{proof}
First observe that if $f$ has the $(i,j)$-open-to-open property
then, for every subset $A^* \sub A$, $f \rest A^*$ has the $(i,j)$-open-to-open property. Now, by Proposition~\ref{partition:open:to:open}, for any $\lbar a \in
\pr_{>1}(A)$ and $\lbar b \in \pr_{>1}(B)$ there is a partition of $\dom
(f_{\lbar a, \lbar b})$ into $(\lbar a, \lbar b)$-definable subsets $V_1, \ldots, V_n$ such that each $f_{\lbar a, \lbar b} \rest V_i$ has the open-to-open property. From here on it is routine to use compactness to obtain a definable partition $A^{1,1}_1, \ldots, A^{1,1}_m$ of $A$ such that each $f \rest A^{1,1}_i$ has the
$(1,1)$-open-to-open property. Iterating this procedure for each
$(i, j) \in I_{n_1} \times I_{n_2}$ on each piece of the partition obtained in the previous step, we eventually get a partition of $A$ that is as desired.
\end{proof}

A closed ball with $k$ maximal open subballs removed is called a \emph{thin punctured ball}. Of course if $k=0$ then it is just a closed ball and if $k=1$ then it is just a thin annulus.
\begin{lem}\label{thin:ball:closed:match}
Let $f : A \fun B$ be a definable bijection that has the open-to-open property, where $A$ and $B$ are disjoint unions of thin punctured balls. Then $A$ and $B$ have the same number of positive boolean components and the same number of negative boolean components.
\end{lem}
\begin{proof}
Let $\ga_1, \ldots, \ga_n$ be the positive boolean components of $A$ and $\gb_1, \ldots, \gb_m$ the positive boolean components of $B$. If $\gm \sub \ga_i$ is a maximal open subball of $\ga_i$ then $f(\gm)$ must be a maximal open subball of some $\gb_j$, and vice versa. So there is a parametrically $\lan{RV}$-definable bijection $g$ from a definable cofinite subset of $\bigcup_{1 \leq i \leq n} \K \times \{i\}$ into a definable cofinite subset of $\bigcup_{1 \leq i \leq m} \K \times \{i\}$. Note that $g$ is a constructible function in the sense of algebraic geometry (see~\cite[Remark~7.5]{Yin:special:trans}). If $n < m$ then we would be able to extend $g$, possibly with additional parameters, to a constructible injective but not surjective function from $\bigcup_{1 \leq i \leq m} \K \times \{i\}$ into itself, violating Ax's theorem (see~\cite[3.G$^\prime$]{gromov:symbolic:var}); similarly if $n > m$. So $n = m$.

By strong minimality, for each $i \leq n$ there are a finite subset $C_i \sub \K \times \{i\}$ and a $j \leq n$ such that $g \rest (\K \times \{i\} \mi C_i)$ is a regular injection into $\K \times \{j\}$. By the conservation property~~\cite[3.G$^{\prime \prime}$]{gromov:symbolic:var}, $C_i$ and $\K \times \{j\} \mi g(\K \times \{i\} \mi C_i)$ have the same size. Since $\bigcup_i C_i$ is finite, it follows that $A$ and $B$ must have the same number of holes.
\end{proof}

The following proposition somehow completes Lemma~\ref{atomic:balls:no:alg:relation}: it justifies the feeling that atomic open balls, atomic closed balls, and atomic thin annuli are geometrically distinct.

\begin{prop}\label{no:bijec:atomic:primi}
There can be no definable bijection between any pair of the following: an atomic open ball, an atomic closed ball, and an atomic thin annulus.
\end{prop}
\begin{proof}
The case of an atomic closed ball and an atomic thin annulus follows from Lemma~\ref{complete:type:open:ball} and Lemma~\ref{thin:ball:closed:match}. The other two cases follow from Lemma~\ref{atomic:balls:no:alg:relation}.
\end{proof}

In fact this proposition is true without the atomicity condition. But even that is a special case that follows from a major construction in the Hrushovski-Kazhdan theory, namely the Euler characteristics. This construction in effect fuses together two Euler characteristics: that of the theory of algebraically closed fields of characteristic $0$ and that of the theory of densely ordered abelian groups. The former is essentially what is used in the proof of Lemma~\ref{thin:ball:closed:match}. Similarly we may use the latter in certain simple situations, for example:

\begin{prop}\label{exam:top:inv}
Let $f : \ga \fun \gb$ be a definable bijection, where $\ga$ is a punctured closed ball with one open hole and $\gb$ is a punctured open ball with one open hole. Then $f$ cannot have the open-to-open property.
\end{prop}
\begin{proof}
Suppose for contradiction that $f$ is an open-to-open function. Let $a$ be a point in the hole of $\ga$ and $b$ a point in the hole of $\gb$. Set $A = \rv(\ga - a)$ and $B = \rv(\gb - b)$. Then $f$ induces a parametrically definable bijection $f' : A \fun B$. By strong minimality, for each $\gamma \in \vrv(A)$ there is a unique $f''(\gamma) \in \vrv(B)$ such that
\[
C_{\gamma} = \vrv^{-1}(f''(\gamma)) \mi f'(\vrv^{-1}(\gamma))
\]
is finite. By compactness the sizes of these finite subsets are bounded. Therefore $f'' : \vrv(A) \fun \vrv(B)$ is a parametrically definable bijection. Note that $f''$ is parametrically definable in $\Gamma$. There are three cases in accordance with the lengths of the half-open interval $\vrv(A)$ and the open interval $\vrv(B)$. Any of these cases would contradict the fact that the Grothendieck ring of $\Gamma$ is nontrivial (see~\cite{kage:fujita:2006}).
\end{proof}

The construction of the Euler characteristics in the Hrushovski-Kazhdan theory will be presented in a sequel.

\section{$2$-cells}\label{section:2unit}

From now on we are back in $\gC$.

The notion of a $2$-cell is studied in this section. It is crucial for our Fubini's theorem (see Lemma~\ref{contraction:perm:pair:isp}).

\begin{lem}\label{inverse:special:dim:1}
Let $A \sub \VF \times \RV^m$ be a definable subset and $T$ a special bijection on $A$ such that $T(A) = A^{\sharp}$ is an $\RV$-pullback. Then there is a definable function $\epsilon : \prv (A^{\sharp}) \fun \VF$ such that, for every $\rv$-polydisc $\rv^{-1}(t) \times \set{(t, \lbar s)} \sub A^{\sharp}$, we have
\[
(\pvf \circ T^{-1})(\rv^{-1}(t) \times \set{(t, \lbar s)}) =
\rv^{-1}(t) + \epsilon(t, \lbar s).
\]
\end{lem}
\begin{proof}
We do induction on the length $\lh(T)$ of $T$. For the base case $\lh(T)=1$, let $T = \can \circ \eta$, where $\eta$ is a centripetal
transformation. Let $\lambda$ and $C \sub \RVH(A)$ be the corresponding focus map and its locus. For each $(t, \lbar s) \in \prv (A^{\sharp})$, if $\lbar s \in \dom(\lambda)$ then set
$\epsilon(t, \lbar s)= \lambda(\lbar s)$, otherwise set
$\epsilon(t, \lbar s) = 0$. Clearly $\epsilon$ is as required.

We proceed to the inductive step. Let $T = \can \circ \eta_n \circ \cdots \circ \can \circ \eta_1$ and $T_1 = \can \circ \eta_{n} \circ \cdots \circ \can \circ \eta_2$. By the inductive hypothesis, for the special bijection $T_1$, there is a
function
\[
\epsilon_1 : (\prv \circ T_1)((\can \circ \eta_1)(A))
\fun \VF
\]
as required. Let $\lambda$ and $C \sub \RVH(A)$ be the focus map and its locus for the centripetal transformation $\eta_1$. For each $(t, \lbar s) \in \prv(A^{\sharp})$, if
\[
(\prv \circ T_1^{-1})(\rv^{-1}(t) \times \set{(t, \lbar s)}) = (r, \lbar u) \quad \text{and} \quad \lbar u \in \dom(\lambda),
\]
then set $\epsilon(t, \lbar s)= \epsilon_1(t, \lbar s) + \lambda(\lbar u)$, otherwise set
$\epsilon(t, \lbar s) = \epsilon_1(t, \lbar s)$. Then $\epsilon$
is as required.
\end{proof}

Note that, in the above lemma, since $\dom(\epsilon) \sub \RV^l$ for some $l$, by
Lemma~\ref{function:rv:to:vf:finite:image}, $\ran(\epsilon)$ is actually finite.

It is easy to see that for functions between subsets that have only one $\VF$-coordinate, composing with special bijections on the right and inverses of special bijections on the left preserves the open-to-open property.

\begin{lem}\label{bijection:dim:1:decom:RV}
Let $A, B \sub \VF$ and $f : A \fun B$ a definable bijection. Then there is a special bijection $T$ on $A$ such that $T(A)$ is an $\RV$-pullback and, for each $\rv$-polydisc $\gp \sub T(A)$, $f \rest T^{-1}(\gp)$ is $\Gamma$-linear and has the open-to-open property.
\end{lem}
\begin{proof}
By Proposition~\ref{partition:open:to:open}, there is a finite partition of $A$ into definable subsets such that the restriction of $f$ to each
piece has the open-to-open property. After applying
Corollary~\ref{all:subsets:rvproduct} to each piece, we may assume that $A$ is an open ball and $f$ has the open-to-open property. For each $a \in A$ let $\delta_a \in \Gamma \cup \{\infty\}$ be the least such that $f \rest \go(a, \delta_a)$ is $\Gamma$-linear. Let $\psi$ be a quantifier-free formula that defines the function $a \efun \delta_a$. By Theorem~\ref{special:bi:polynomial:constant} and compactness, there is a special bijection $T$ on $A$ such that $T(A)$ is an $\RV$-pullback and each term of the form $\rv(G(X))$ in $\psi$ is constant on every subset of the form $T^{-1}(\gp)$, where $\gp$ is an $\rv$-polydisc contained in $T(A)$. By Corollary~\ref{rv:linear:rad}, each $f \rest T^{-1}(\gp)$ is $\Gamma$-linear. So $T$ is as required.
\end{proof}


\begin{lem}\label{bijection:radii:one:one}
Let $A$, $B \sub \VF$ be open balls and $f : A \fun B$ a
definable bijection that has the open-to-open property. For any $\alpha \in \Gamma(S)$ there is a special bijection $T$ on $A$ such that $T(A)$ is an $\RV$-pullback and, for each $\rv$-polydisc $\gp \sub T(A)$, the set
\[
\set{\rad(\gb) : \gb \text{ is an open ball contained in }
T^{-1}(\gp) \text{ with } \rad(f(\gb)) = \alpha}
\]
is a singleton.
\end{lem}
\begin{proof}
Let $h$ be the function on $A$ such that, for every $a \in A$, if $\go(f(a), \alpha) \sub B$ then $h(a) = \rad(f^{-1}(\go(f(a), \alpha)))$, otherwise $h(a) = \infty$. Now we may proceed exactly as in Lemma~\ref{bijection:dim:1:decom:RV}.
\end{proof}

\begin{lem}\label{bijection:rv:one:one}
Let $A \sub \VF^2$ be a definable subset such that $\pr_1(A)$ is an
open ball. Let $f : \pr_1(A) \fun \pr_2(A)$ be a definable bijection that has the open-to-open property. Suppose that for each $a \in \pr_1(A)$ there is a $t_a \in \RV$ such that
\[
\fib(A, a) = \rv^{-1}(t_a) + f(a).
\]
Then there is a special bijection $T$ on $\pr_1(A)$ such that $T(\pr_1(A))$ is an $\RV$-pullback and, for each $\rv$-polydisc $\gp \sub T(\pr_1(A))$, the set
\[
\set{\rv(a - f^{-1}(b)) : a \in T^{-1}(\gp) \text{ and } b \in \fib(A, a)}
\]
is a singleton.
\end{lem}
\begin{proof}
For each $a \in \pr_1(A)$, let $\gb_a$ be the minimal closed ball
that contains $\fib(A, a)$. Since $\fib(A, a) - f(a) = \rv^{-1}(t_a)$, we have $f(a) \in \gb_a$ but $f(a) \notin \fib(A, a)$ if $t_a \neq \infty$. Hence $a \notin
f^{-1}(\fib(A, a))$ if $t_a \neq \infty$ and $\set{a} =
f^{-1}(\fib(A, a))$ if $t_a = \infty$. Since $f^{-1}(\fib(A, a))$ is a ball, in either case, the function $\rv(a - X)$ is constant on $f^{-1}(\fib(A, a))$. The function $h : \pr_1 (A) \fun \RV$ given by
\[
a \efun \rv(a - f^{-1}(\fib(A, a)))
\]
is definable. Now we may proceed as in Lemma~\ref{bijection:dim:1:decom:RV}.
\end{proof}

%

\begin{defn}\label{defn:balance}
Let $A \sub \VF^2$ be a definable subset such that $\pr_1(A)$, $\pr_2(A)$ are open balls. Let $f : \pr_1(A) \fun \pr_2(A)$ be an $\Gamma$-linear open-to-open bijection. We say that $f$ is \emph{balanced in $A$} if there
are $t_1, t_2 \in \RV$ such that, for each $a \in \pr_1(A)$,
\[
 \fib(A, a) = \rv^{-1}(t_2) + f(a), \quad
 f^{-1}(\fib(A, a)) = a - \rv^{-1}(t_1).
\]
The elements $t_1$, $t_2$ are called the \emph{paradigms} of $A$.
\end{defn}

Let $f$ be balanced in $A$ with the paradigms $t_1$, $t_2$. If one of the paradigms is $\infty$ then the other one must be $\infty$. In this case $A$ is the (graph of the) function $f$. Let us suppose that $t_1, t_2 \in \RV^{\times}$. Let $a \in \pr_1(A)$, $\gb$ the minimal closed ball containing
$\fib(A, a)$, and $\ga$ the minimal closed ball containing
$f^{-1}(\fib(A, a))$. The following properties are easily
deduced:
\begin{enumerate}
 \item $f(a) \notin \fib(A, a)$ and hence $a \notin f^{-1}(\fib(A, a))$.

 \item $\vrv(t_1) = \rad(\ga) > \rad(\pr_1(A))$ and $\vrv(t_2) = \rad(\gb) > \rad(\pr_2(A))$.

 \item $f(a) \in \gb$ and $a \in \ga$.

 \item Let $\go_{a}$, $\go_{f(a)}$ be the maximal open subballs of $\ga$, $\gb$ that contains respectively $a$, $f(a)$. We have that, for every $a^* \in f^{-1}(\go_{f(a)})$,
     \[
     \fib(A, a^*) = \rv^{-1}(t_2) + f(a^*) = \rv^{-1}(t_2) + f(a) = \fib(A, a)
     \]
     and hence $a^* - \rv^{-1}(t_1) = a - \rv^{-1}(t_1)$; so $a^* \in \go_{a}$. Symmetrically, for every $b^* \in f(\go_a)$,
     \[
     f^{-1}(\fib(A, f^{-1}(b^*))) = f^{-1}(b^*) - \rv^{-1}(t_1) = a - \rv^{-1}(t_1) = f^{-1}(\fib(A, a))
     \]
     and hence $\rv^{-1}(t_2) + b^* = \rv^{-1}(t_2) + f(a)$; so $b^* \in \go_{f(a)}$. Therefore we have $f(\go_{a}) = \go_{f(a)}$.
\end{enumerate}
Let $\gA$, $\gB$ be respectively the sets of open subballs of $\pr_1(A)$, $\pr_2(A)$ of radii $\vrv(t_1)$, $\vrv(t_2)$. Then we also have:
\begin{enumerate}
 \item $f$ induces a bijection $f_{\downarrow} : \gA \fun \gB$ such that for each $\go \in \gA$, each $c \in \go$, and each $d \in \fib(A, c)$
     \[
     \fib(A, c) = \rv^{-1}(t_2) + f_{\downarrow}(\go) \in \gB, \quad
     \fib(A, d) = f_{\downarrow}^{-1}(\fib(A, c)) + \rv^{-1}(t_1) = \go.
     \]
 \item There is an internal symmetry of $A$ as expressed by the following identity:
       \begin{align*}
       \bigcup \{\go \times (\rv^{-1}(t_2) + f_{\downarrow}(\go)) : \go \in \gA \} &= \bigcup \{(f^{-1}_{\downarrow}(\gm) + \rv^{-1}(t_1)) \times \gm : \gm \in \gB \} \\
       & = A \cap (\pr_1(A) \times \pr_2(A)).
       \end{align*}
\end{enumerate}

\begin{rem}\label{rem:rv:linear:bal}
Since $f$ is also $\Gamma$-linear, there is a $\gamma \in \Gamma$ such that $\vv(f(a) - f(a')) = \gamma + \vv(a - a')$ for any $a, a' \in \pr_1(A)$. In particular, this holds for any $a, a' \in \go \in \gA$ and hence $\gamma = \vrv(t_2/t_1)$. More generally, let $T_1$, $T_2$ be special bijections on $\pr_1(A)$, $\pr_2(A)$ such that $T_1(\pr_1(A))$, $T_2(\pr_2(A))$ are $\RV$-pullbacks. Let $\rv^{-1}(r_1) \times \{\lbar s_1\} \sub T_1(\pr_1(A))$ and suppose that
\[
(T_2 \circ f \circ T^{-1}_1)(\rv^{-1}(r_1) \times \{\lbar s_1\}) = \rv^{-1}(r_2) \times \{\lbar s_2\} \sub T_2(\pr_2(A)).
\]
Then $\vrv(r_2 / r_1) = \gamma = \vrv(t_2/t_1)$. This consequence of the $\Gamma$-linearity of $f$ is what will be used below in the last section.
\end{rem}

\begin{defn}\label{def:units}
We say that a subset $A$ is a \emph{$1$-cell} if it is either an
open ball contained in one $\rv$-ball or a point in $\VF$. We say
that $A$ is a \emph{$2$-cell} if
\begin{enumerate}
 \item $A \sub \VF^2$ is contained in one $\rv$-polydisc and
  $\pr_1(A)$ is a $1$-cell,
 \item there is a function $\epsilon : \pr_1 (A) \fun \VF$ and a $t \in
   \RV$ such that, for each $a \in \pr_1(A)$, $\fib(A, a) =
   \rv^{-1}(t) + \epsilon(a)$,
 \item one of the following three possibilities occurs:
  \begin{enumerate}
   \item $\epsilon$ is constant,
   \item $\epsilon$ is injective, has the open-to-open property, and
    $\rad(\epsilon(\pr_1(A))) \geq \vrv(t)$,
   \item $\epsilon$ is balanced in $A$.
  \end{enumerate}
\end{enumerate}
The function $\epsilon$ is called the \emph{positioning function}
of $A$ and the element $t$ is called the \emph{paradigm} of $A$.
\end{defn}

A subset $A \sub \VF \times \RV^m$ is a \emph{$1$-cell} if for each $\lbar t \in \prv(A)$ the fiber $\fib(A, \lbar t)$ is a $1$-cell in the sense of Definition~\ref{def:units}. The parameterized version of the notion of a $2$-cell is formulated in the same way. A cell is definable if all the relevant ingredients are definable. Naturally we shall only be concerned with definable cells.

Suppose that $A$ is a $2$-cell. Clearly if its paradigm $t$ is $\infty$ then $A$ and its positioning function $\epsilon$ are identical. It is also easy to see that, if $t \neq \infty$ and $\epsilon$ is not balanced, then $A$ is actually an open
polydisc.

Notice that Corollary~\ref{all:subsets:rvproduct} implies that
for every definable subset $A \sub \VF \times \RV^m$ there is a
definable function $P: A \fun \RV^l$ such that, for each $\lbar s
\in \ran(P)$, the fiber $P^{-1}(\lbar s)$ is a $1$-cell. The same
holds for $2$-cell:

\begin{lem}\label{decom:into:2:units}
For every definable subset $A \sub \VF^2$ there is a definable
function $P: A \fun \RV^m$ such that, for each $\lbar s \in \ran(P)$, the fiber $P^{-1}(\lbar s)$ is a $2$-cell.
\end{lem}
\begin{proof}
By compactness, without loss of generality, we may assume that $A$ is contained in
one $\rv$-polydisc. For any $a \in \pr_1 (A)$, by Corollary~\ref{all:subsets:rvproduct}, there is an $a$-definable special bijection $T_a$ on $\fib(A, a)$ such that $T_a(\fib(A, a))$ is an $\RV$-pullback. By Lemma~\ref{inverse:special:dim:1}, there is an $a$-definable function $\epsilon_a : \prv (T_a(\fib(A, a))) \fun \VF$ such that, for every $(t, \lbar s)
\in \prv (T_a(\fib(A, a)))$, we have
\[
T_a^{-1}(\rv^{-1}(t) \times \set{(t, \lbar s)}) =
\rv^{-1}(t) + \epsilon_a(t, \lbar s).
\]
By compactness, we may glue these functions together, that is, there is a definable subset $A^{\sharp} \sub \pr_1(A) \times \RV^l$ and a definable function $\epsilon : A^{\sharp} \fun \VF$ such that, for each $a \in \pr_1(A)$,
\[
\fib(A^{\sharp}, a) = \prv (T_a(\fib(A, a))), \quad
\epsilon \rest \fib(A^{\sharp}, a) = \epsilon_a.
\]
Since, for each $(t, \lbar s) \in \prv(A^{\sharp})$, $\epsilon \rest \fib(A^{\sharp}, (t, \lbar s))$ is a $(t, \lbar s)$-definable function from $\VF$ into $\VF$, by Lemma~\ref{function:dim1:decom:RV} and compactness, we are reduced to the case
that each $\epsilon \rest \fib(A^{\sharp}, (t, \lbar s))$ is either
constant or injective. If no $\epsilon \rest \fib(A^{\sharp}, (t, \lbar
s))$ is injective then we can finish by applying
Corollary~\ref{all:subsets:rvproduct} to each $\fib(A^{\sharp}, (t,
\lbar s))$ and then compactness.

Suppose that $\epsilon_{(t, \lbar s)} = \epsilon \rest \fib(A^{\sharp},
(t, \lbar s))$ is injective. By Lemma~\ref{bijection:dim:1:decom:RV}, we are reduced to the case that $\fib(A^{\sharp}, (t, \lbar s))$ is an open ball and
$\epsilon_{(t, \lbar s)}$ is $\Gamma$-linear and has the open-to-open property. Note that if $\rad(\ran (\epsilon_{(t, \lbar s)})) \geq \vrv(t)$ then $\epsilon_{(t, \lbar s)}$ satisfies the condition (3b) in Definition~\ref{def:units}. So let us suppose $\rad(\ran (\epsilon_{(t, \lbar s)})) < \vrv(t)$. We have
\[
\ran (\epsilon_{(t, \lbar s)}) = \bigcup_{a \in \fib(A^{\sharp},
(t, \lbar s))} (\rv^{-1}(t) + \epsilon_{(t, \lbar s)}(a)).
\]
By Lemma~\ref{bijection:rv:one:one}, we are further reduced to the case that there is an $r \in \RV$ such that, for every $a \in \fib(A^{\sharp}, (t, \lbar s))$,
\[
\rv(a - \epsilon_{(t, \lbar s)}^{-1}(\rv^{-1}(t) + \epsilon_{(t, \lbar s)}(a))) = r
\]
and hence
\[
\epsilon_{(t, \lbar s)}^{-1}(\rv^{-1}(t) + \epsilon_{(t, \lbar s)}(a)) = a -
\rv^{-1}(r).
\]
So, in this case, $\epsilon_{(t, \lbar s)}$ is balanced. Now we are
done by compactness.
\end{proof}

\section{Contracting to $\RV$}\label{section:contracting}

One basic result of this section is Lemma~\ref{simul:special:dim:1}. We remark here again that if we could extend this lemma to higher dimensions in a more ``natural'' way then the length of this paper would be cut by half.

\begin{lem}\label{bijection:made:contractible}
Let $A \sub \VF^{n_1} \times \RV^{m_1}$ and $B \sub \VF^{n_2} \times \RV^{m_2}$ be definable subsets and $f : A \fun B$ a definable function. Then there exist special bijections $T_A$, $T_B$ on $A$, $B$ such that $T_A(A)$, $T_B(B)$ are $\RV$-pullbacks and the function $T_B \circ f \circ T_A^{-1}$ is contractible.
\end{lem}
\begin{proof}
Note that by our convention the canonical bijection has been applied to all definable subsets. By Corollary~\ref{all:subsets:rvproduct}, there is a special bijection $T_B$ on $B$ such that $T_B(B)$ is an $\RV$-pullback. So we may assume that $B$ is an $\RV$-pullback. Let $\psi_1, \ldots,
\psi_{m_2}$ be a sequence of quantifier-free formulas that respectively define the functions $f_i = \pr_i \circ \prv \circ f$ for $1 \leq i \leq m_2$. Let $G_i(\lbar X)$ enumerate all the occurring polynomials in $\psi_1,
\ldots, \psi_{m_2}$. By Theorem~\ref{special:bi:polynomial:constant}, there is a special bijection $T_A$ on $A$ such that $T_A(A)$ is an
$\RV$-pullback and each function $\rv \circ G_i \circ T_A^{-1}$ is
constant on every $\rv$-polydisc $\gp \sub T_A(A)$. So on such an $\rv$-polydisc every $f_i \circ T_A^{-1}$ is constant.
\end{proof}

\begin{lem}\label{simul:special:dim:1}
Let $A \sub \VF \times \RV^{m_1}$ and $B \sub \VF \times \RV^{m_2}$ be definable subsets and $f : A \fun B$ a definable bijection. Then there exist special bijections $T_A : A \fun A^{\sharp}$ and $T_B : B \fun B^{\sharp}$ such that $A^{\sharp}$, $B^{\sharp}$ are $\RV$-pullbacks and, in
the commutative diagram
\[
\bfig
  \square(0,0)/->`->`->`->/<600,400>[A`A^{\sharp}`B`B^{\sharp};
  T_A`f``T_B]
 \square(600,0)/->`->`->`->/<600,400>[A^{\sharp}`\rv(A^{\sharp})`B^{\sharp} `\rv(B^{\sharp});  \rv`f^{\sharp}`f^{\sharp}_{\downarrow}`\rv]
 \efig
\]
$f^{\sharp}_{\downarrow}$ is bijective and hence $f^{\sharp}$ is a lift of it.
\end{lem}
\begin{proof}
By Proposition~\ref{open:to:open:parameter}, there is a finite partition of $A$ into definable subsets $A_1, \ldots, A_n$ such that each $f \rest A_i$ has the open-to-open property. Therefore, applying
Lemma~\ref{bijection:made:contractible} to each $f\rest A_i$, we may assume that $A$, $B$ are $\RV$-pullbacks and $f$ is contractible and has the open-to-open property. In particular, for each $\rv$-polydisc $\gp \sub A$, $f(\gp)$ is an
open polydisc contained in an $\rv$-polydisc. By Lemma~\ref{bijection:made:contractible} again, there is a
special bijection $T_B: B \fun B^{\sharp}$ such that $(T_B \circ f)^{-1}$ is contractible. Let $T_B = \can \circ \eta_n \circ \ldots \circ \can \circ \eta_1$, where each $\eta_i$ is a centripetal transformation and $\can$ is the canonical bijection.

Now, by induction, we shall construct a special bijection $T_A = \can \circ \eta_n^* \circ \ldots \circ \can \circ \eta_1^*$ on $A$ such that, for each $i$, both $L_i \circ f \circ (L_i^*)^{-1}$ and $ L_i^* \circ (T_B \circ f)^{-1}$ are contractible, where
\[
L_i = \can \circ \eta_i \circ \ldots \circ \can \circ \eta_1,\quad
L_i^* = \can \circ \eta_i^* \circ \ldots \circ \can \circ \eta_1^*.
\]
Then $T_A$, $T_B$ will be as desired. To that end, suppose that $\eta^*_i$ has been constructed for each $i \leq k < n$. Let $D^*_k = L_k^*(A)$ and $D_k = L_k(B)$. Let $C \sub D_k$ be the locus of $\eta_{k+1}$ and $\lambda$ the corresponding focus map. Since $L_k \circ f \circ (L_k^*)^{-1}$ is contractible and has the open-to-open property, each $\rv$-polydisc $\gp \sub D_k$ is the union of disjoint subsets of the form $(L_k \circ f \circ (L_k^*)^{-1})(\gq)$, where $\gq \sub D^*_k$ is an $\rv$-polydisc. For each $\lbar t = (t_1, \lbar t_1) \in \dom(\lambda)$, let
\[
O_{\lbar t} = \set{\gq \sub D^*_k : \gq \text{ is an $\rv$-polydisc and }
(L_k \circ f \circ (L_k^*)^{-1})(\gq) \sub \rv^{-1}(t_1) \times
\{\lbar t \}}.
\]
Then there is a unique open sub-polydisc $\go_{\lbar t} \sub \rv^{-1}(t_1) \times
\{\lbar t \} \sub C$ such that $(\lambda(\lbar t), \lbar t) \in \go_{\lbar t}$ and there is a unique $\gq_{\lbar t} \in O_{\lbar t}$ with $(L_k \circ f \circ (L_k^*)^{-1})(\gq_{\lbar t}) = \go_{\lbar t}$. Let
\[
C^* = \bigcup_{\lbar t \in \dom(\lambda)} \gq_{\lbar t} \sub D^*_k, \quad
a_{\lbar t} = (L_k \circ f \circ (L_k^*)^{-1})^{-1}(\lambda(\lbar t),
\lbar t) \in \gq_{\lbar t}.
\]
Let $\lambda^* : \pr_{>1} (C^*) \fun \VF$ be the corresponding focus
map given by $\pr_{>1} (\gq_{\lbar t}) \efun a_{\lbar t}$. Note that both $C^*$ and $\lambda^*$ are definable. Let $\eta^*_{k+1}$ be the centripetal transformation determined by $C^*$ and $\lambda^*$. For each $\lbar t \in \dom(\lambda)$, the restriction of $L_{k+1} \circ f \circ (L_{k+1}^*)^{-1}$ to
$\can(\gq_{\lbar t} - a_{\lbar t})$ is a bijection between the
$\RV$-pullbacks $\can(\gq_{\lbar t} - a_{\lbar t})$ and
$\can(\go_{\lbar t} - \lambda(\lbar t))$ that is contractible in
both ways. For any $\gq \in O_{\lbar t}$ with $\gq \neq \gq_{\lbar t}$,
\[
(L_{k+1} \circ f \circ (L_{k+1}^*)^{-1})(\can(\gq))
\]
is an open polydisc contained in an $\rv$-polydisc. Therefore $L_{k+1} \circ f \circ (L_{k+1}^*)^{-1}$ is contractible.

On the other hand, it is easy to see that, for any nondegenerate $\rv$-polydisc $\gp \sub B^{\sharp}$, $L_{k}^* \circ (T_B \circ f)^{-1}(\gp)$ does not contain any $a_{\lbar t}$ and hence, by the construction of $L_k^*$, $L_{k+1}^* \circ (T_B \circ f)^{-1}$ is contractible.
\end{proof}

\begin{rem}\label{special:dim:1:RV:iso}
In the above lemma, $(\rv(A^{\sharp}), \pr_1)$, $(\rv(B^{\sharp}), \pr_1)$ are actually $\RV[1,\cdot]$-isomorphic. This is immediate by Remark~\ref{RV:isomorphism:weaker:condition}. In fact, for any $(a, \lbar t) \in A^{\sharp}$, if $f^{\sharp}(a, \lbar t) = (b, \lbar s)$ then, by Lemma~\ref{dcl:to:ac}, $b$ is $a$-algebraic and hence, by Lemma~\ref{acl:VF:transfer:acl:RV}, $\rv(b)$ is $\rv(a)$-algebraic; similarly for the other direction.
\end{rem}


\begin{defn}
Let $A \sub \VF^{n} \times \RV^{m_1}$ and $B \sub \VF^{n} \times \RV^{m_2}$ and $f : A \fun B$ a bijection. We say that $f$ is \emph{relatively unary} if there is an $i \in I_n$ such that $(\pr_{\widetilde{i}} \circ f)(\lbar x) = \pr_{\widetilde{i}}(\lbar x)$ for all $\lbar x \in A$. In this case we say that $f$ is \emph{unary relative to the coordinate $i$}. If $f \rest \fib(A, \lbar a)$ is also a special bijection for every $\lbar a \in \pr_{\widetilde{i}} (A)$ then we say that $f$ is \emph{special relative to the coordinate $i$}.
\end{defn}

Obviously the inverse of a relatively unary bijection is a relatively unary bijection.

Let $A \sub \VF^n \times \RV^m$ be a definable subset, $C \sub \RVH(A)$ an $\RV$-pullback, $\lambda$ a focus map with respect to $C$ (and the coordinate $1$), and $\eta$ the centripetal transformation with respect to $\lambda$. Clearly $\eta$ is unary relative to the coordinate $1$. It follows that every special bijection $T$ on $A$ is a composition of relatively special bijections. Choose an $i \in I_n$. By Corollary~\ref{all:subsets:rvproduct} and compactness, there is a bijection $T_i$ on $A$, special relative to the coordinate $i$, such that $T_i(\fib(A, \lbar a))$ is an $\RV$-pullback for every $\lbar a \in \pr_{\widetilde i}(A)$. Let
\[
A_i = \bigcup_{\lbar a \in \pr_{\widetilde i}(A)} \set{\lbar a} \times (\prv \circ T_i)(\fib(A, \lbar a)) \sub \VF^{n-1} \times \RV^{m_i}.
\]
We write $\wh T_i : A \fun A_i$ for the function naturally induced by $T_i$. For any $j \in I_{n-1}$, we may repeat the above procedure on $A_i$ with respect to the coordinate $j$ and obtain a subset $A_{j} \sub \VF^{n-2} \times \RV^{m_j}$ and a function $\wh T_{j} : A_i \fun A_{j}$. Continuing this procedure, we see that, for any permutation $\sigma$ of $I_n$, we can construct a (not necessarily unique) sequence of relatively special bijections $T_{\sigma(1)}, \ldots, T_{\sigma(n)}$ and a corresponding function $\wh T_{\sigma} : A \fun \RV^{l}$. We also have the natural bijection determined by $\wh T_{\sigma}$:
\[
T_{\sigma} : A \fun \mathbb{L}(\wh T_{\sigma}(A), \pr_{\leq n}),
\]
where, without loss of generality, it is (always) assumed that the relevant coordinates in $\wh T_{\sigma}(A)$ are the first $n$ ones. Note that $T_{\sigma}$ is a special bijection and may be thought of as the ``composition'' $T_{\sigma(n)} \circ \ldots \circ T_{\sigma(1)}$.


\begin{defn}\label{defn:standard:contraction}
The function $\wh T_{\sigma}$ (or the image $\wh T_{\sigma}(A)$) is called a \emph{standard contraction} of $A$.
\end{defn}

Let $\wh T_{\id}$ be a standard contraction of $A$ such that $T_{\id}(A)$ ($= (T_n \circ \ldots \circ T_1)(A)$) is of the form $\rv^{-1}(t_i) \times \{(\lbar 0, t_i, \lbar \infty, \lbar s) \}$, where $\lbar 0$ is a tuple of $0$ of length $n-1$ and $\lbar \infty$ is the corresponding tuple of $\infty$ of length $n-1$. Let $T_{\leq i} = T_i \circ \ldots \circ T_1$. It is not hard to see that $T_{\leq i}(A)$ is of the form $\rv^{-1}(t_i) \times \{(\lbar 0, \lbar a, t_i, \lbar \infty, \lbar s) \}$, where $\lbar 0$ is a tuple of $0$ of length $i-1$, $\lbar a \in \VF$ is
a tuple of length $n - i$, and $\lbar \infty$ is a tuple of
$\infty$ of length $i-1$. So for any distinct $a, b \in \rv^{-1}(t_i)$ we have
\[
(\pvf_i \circ T_{\id}^{-1})(a, \lbar 0, t_i, \lbar \infty,
\lbar s) \neq (\pvf_i \circ T_{\id}^{-1})(b, \lbar 0, t_i,
\lbar \infty, \lbar s).
\]
This simple observation is used to prove the following:

\begin{lem}\label{bijection:partitioned:unary}
Let $A \sub \VF^{n} \times \RV^{m_1}$, $B \sub \VF^{n} \times
\RV^{m_2}$, and $f : A \fun B$ a definable bijection. Then there is a partition of $A$ into definable subsets $A_i$ such that each $f \rest A_i$ is a composition of definable relatively unary bijections.
\end{lem}
\begin{proof}
We do induction on $n$. Since the base case $n=1$ holds vacuously,
we proceed to the inductive step directly. By Corollary~\ref{all:subsets:rvproduct} and compactness, for each $\lbar a = (a_1, \ldots, a_{n-1}) \in \pvf_{< n}(A)$, there is an $\lbar a$-definable standard contraction $\wh T_{\id, \lbar a}$ on $f(\fib(A , \lbar a))$ such that $(T_{\id, \lbar a} \circ
f)(\fib(A , \lbar a)) = Z_{\lbar a}$ is an $\RV$-pullback. By Lemma~\ref{VF:dim:rv:polydisc}, in each tuple $(\lbar t, \lbar s) = (t_1, \ldots, t_n, \lbar s) \in \prv(Z_{\lbar a})$, there is at most one $i \leq n$ such that $t_i \neq \infty$, that is, each $\rv$-polydisc contained in $Z_{\lbar a}$ is of the form $\rv^{-1}(t_i) \times \{(\lbar 0, t_i, \lbar \infty, \lbar s) \}$
for some $i \leq n$. So there is a partition of $\fib(A , \lbar a)$ into $\lbar a$-definable subsets $A_{\lbar a}^1, \ldots, A_{\lbar a}^n$ such that if $(\lbar a, a_n, \lbar r) \in A_{\lbar a}^i$ then $(T_{\id, \lbar a} \circ f)(\lbar a,
a_n, \lbar r)$ is of the form $(b_i, \lbar 0, t_i, \lbar \infty, \lbar s)$. By the observation above, if $b, b' \in \rv^{-1}(t_i)$ are distinct then
\[
(\pvf_i \circ T_{\id, \lbar a}^{-1})(b_i, \lbar 0, t_i, \lbar
\infty, \lbar s) \neq (\pvf_i \circ T_{\id, \lbar a}^{-1})(b'_i, \lbar
0, t_i, \lbar \infty, \lbar s).
\]
Let $g_{\lbar a, i}$ be the function on $A_{\lbar a}^i$ given by
\[
(\lbar a, a_n, \lbar r) \efun (\lbar a, d_i, \lbar r, t_i,
\lbar \infty, \lbar s),
\]
where
\[
(\prv \circ T_{\id, \lbar a} \circ f)(\lbar a, a_n, \lbar r) = (t_i, \lbar \infty, \lbar s), \quad (\pvf_i \circ f)(\lbar a, a_n, \lbar r) = d_i.
\]
Therefore, after reindexing the $\VF$-coordinates in each $A_{\lbar a}^i$ separately, each $g_{\lbar a, i}$ is an $\lbar a$-definable unary bijection on $A_{\lbar a}^i$ relative to the coordinate $i$ such that $\pvf_i \circ f = \pvf_i \circ g_{\lbar a, i}$. By compactness, there are a partition of $A$ into definable subsets $A_1, \ldots, A_n$ and definable unary bijections $g_i$ on $A_i$ relative to the coordinate $i$ such that $\pvf_i \circ f = \pvf_i \circ g_i$.

For each $i \leq n$ let $h_i$ be the function on $g_i(A_i)$ such that $f \rest A_i = h_i \circ g_i$. For each $a \in (\pvf_i \circ g_i)(A_i)$, since $h_i(\fib(g_i(A_i), a)) = \fib(f(A_i), a)$, by the inductive hypothesis, there is a partition of $\fib(g_i(A_i), a)$ into $a$-definable subsets $D_a^1, \ldots, D_a^l$ such that each $h_i \rest D_a^j$ is a composition of $a$-definable relatively unary
bijections. So the inductive step holds by compactness.
\end{proof}


Below let $12$, $21$ denote the permutations of $\set{1, 2}$.

\begin{lem}\label{2:unit:contracted}
Let $A \sub \VF^2$ be a definable $2$-cell. Then there are standard contractions $\wh T_{12}$ and $\wh R_{21}$ of $A$ such that there is an $\RV[2, \cdot]$-isomorphism
\[
F : (\wh T_{12}(A), \pr_{\leq 2}) \fun (\wh R_{21}(A), \pr_{\leq 2}).
\]
\end{lem}
\begin{proof}
Let $\epsilon$ be the positioning function of $A$ and $t \in \RV$ the paradigm of $A$. If $t = \infty$ then $A$ is the function $\epsilon : \pr_1(A) \fun \pr_2(A)$, which is either a constant function or a bijection. In the former case, since $A$ is essentially just an open ball, the lemma simply follows from
Corollary~\ref{all:subsets:rvproduct}. In the latter case, there
are special bijections $T_2$, $R_1$ on $A$ relative to the coordinates $2$, $1$ such that
\[
T_2(A) = \pr_1(A) \times \set{(0, \infty)}, \quad
R_1(A) = \set{0} \times \pr_2(A) \times \set{\infty}.
\]
So the lemma follows from Remark~\ref{special:dim:1:RV:iso}. For the rest of the proof we assume $t \neq \infty$.

If $\epsilon$ is not balanced in $A$ then $A = \pr_1(A) \times \pr_2(A)$ is an open polydisc. By Corollary~\ref{all:subsets:rvproduct}, there are special bijections $T_1$, $T_2$ on $\pr_1(A)$, $\pr_2(A)$ such that $T_1(\pr_1(A))$, $T_2(\pr_2(A))$ are $\RV$-pullbacks. In this case the standard contractions determined by $(T_1, T_2)$ and $(T_2, T_1)$ are essentially the same.

Suppose that $\epsilon$ is balanced in $A$. Let $r$ be the other paradigm of $\epsilon$. Recall that $\epsilon : \pr_1 (A) \fun \pr_2(A)$ is again a bijection. Let $T_2$ be the special bijection on $A$ relative to the coordinate $2$ given by $(a, b) \efun (a, b - \epsilon(a))$ and $R_1$ the special bijection on $A$ relative to the coordinate $1$ given by $(a, b) \efun (a - \epsilon^{-1}(b), b)$, where $(a, b) \in A$. Clearly
\[
T_2(A) = \pr_1(A) \times \rv^{-1}(t) \times \set{t},\quad
R_1(A) = \rv^{-1}(r) \times \pr_2(A) \times \set{r}.
\]
So, again, the lemma follows from Remark~\ref{special:dim:1:RV:iso}.
\end{proof}


\begin{cor}\label{subset:partitioned:2:unit:contracted}
Let $A \sub \VF^2 \times \RV^m$ be a definable subset. Then there is a definable injection $f : A \fun \VF^2 \times \RV^l$ such that
\begin{enumerate}
  \item $f$ is unary relative to both coordinates,
  \item there are standard contractions $\wh T_{12}$, $\wh R_{21}$ of $f(A)$ such that $(\wh T_{12}(f(A)), \pr_{\leq 2})$, $(\wh R_{21}(f(A)), \pr_{\leq 2})$ are $\RV[2, \cdot]$-isomorphic.
\end{enumerate}
\end{cor}
\begin{proof}
By Lemma~\ref{decom:into:2:units}, there is a definable function $f: A \fun \VF^2 \times \RV^l$ such that $f(A)$ is a $2$-cell and, for each $(\lbar a, \lbar t) \in A$, $f(\lbar a, \lbar t) = (\lbar a, \lbar t, \lbar s)$ for some $\lbar s \in \RV^{l-m}$. By Lemma~\ref{2:unit:contracted} and compactness, there are standard contractions $\wh T_{12}$ and $\wh R_{21}$ of $f(A)$ into $\RV^{k+l}$ such that there is a commutative diagram:
\[
\bfig
  \Vtriangle(0,0)/->`->`->/<400,400>[\wh T_{12}(f(A))`\wh R_{21}(f(A))`\RV^l; F`\pr_{> k}`\pr_{> k}]
\efig
\]
where $F$ is a definable bijection. By inspection of the proof of Lemma~\ref{2:unit:contracted} and how Lemma~\ref{dcl:to:ac} is used in Remark~\ref{special:dim:1:RV:iso}, we see that $F$ is indeed an $\RV[2, \cdot]$-isomorphism between the objects $(\wh T_{12}(f(A)), \pr_{\leq 2})$ and $(\wh R_{21}(f(A)), \pr_{\leq 2})$.
\end{proof}

\section{The kernel of $\mathbb{L}$ without volume forms}\label{section:kernel}

For notational simplicity, we shall denote by $\mathbb{L}$ all the semigroup (or group) homomorphisms induced by the map $\mathbb{L}$. In this section we shall focus on the categories without volume forms and give a canonical description of the kernel of $\mathbb{L}$. This will yield various invariants of definable bijections that satisfy a kind of Fubini's theorem.

\subsection{Blowups in $\RV$ and the congruence relation $\isp$}

The notion of a special bijection in the $\VF$-sort has a counterpart in the $\RV$-sort, namely that of a blowup:

\begin{defn}\label{defn:blowup}
Let $(U, f) \in \RV[k,\cdot]$ such that $f_{k}(\lbar t) \neq \infty$ and $f_{k}(\lbar t) \in \acl(f_{1}(\lbar t), \ldots, f_{k-1}(\lbar t))$ for all $\lbar t \in U$. Let $(U, f)^{\sharp} = (U^{\sharp}, f^{\sharp}) \in \RV[k,\cdot]$ such that $U^{\sharp} = U \times \RV^{>1}$ and, for any $(\lbar t, s) \in U^{\sharp}$,
\[
f^{\sharp}_{i}(\lbar t, s) = f_{i}(\lbar t) \text{ for }1 \leq i < k, \quad
f^{\sharp}_{k}(\lbar t, s) = s f_{k}(\lbar t).
\]
The object $(U, f)^{\sharp}$ is an \emph{elementary blowup} of $(U, f)$. An elementary blowup of any subobject of $(U, f)$ is an \emph{elementary sub-blowup} of $(U, f)$.

Let $(V, g) \in \RV[k,\cdot]$ and $(C, g \rest C) \in \RV[k,\cdot]$ a subobject of $(V, g)$. Let $F : (U, f) \fun (C, g \rest C)$ be an isomorphism. Then
\[
(U, f)^{\sharp} \uplus (V \mi C, g \rest (V \mi C)) = (U^{\sharp} \uplus (V \mi C),
f^{\sharp} \uplus (g \rest (V \mi C)))
\]
is a \emph{blowup of $(V, g)$ via $F$}, written as $(V, g)^{\sharp}_{F}$, where the subscript $F$ may be dropped when its absence will not cause confusion. The subobject $(C, g \rest C)$ (or the subset $C$) is called the \emph{blowup locus} of $(V, g)^{\sharp}_{F}$. Let $(W, h) \in \RV[k,\cdot]$ be isomorphic to a subobject of $(V, g)$. Then the blowup of $(W, h)$
induced by $F$, that is, the disjoint union of an elementary sub-blowup of $(U, f)$ and a subobject of $(W, h)$, is a \emph{sub-blowup of $(V, g)$ via $F$}.

An \emph{iterated blowup} is a composition of finitely many blowups. The \emph{length} of an iterated blowup is the length of the composition, that is, the number of the blowups involved.
\end{defn}

In the above definition, the condition that the coordinate of interest is algebraically dependent on the other ones is needed for matching blowups with special bijections, since, otherwise, we will not be able to find enough centers of $\rv$-balls to construct focus maps.

Note that if there is an elementary blowup of $(U, f) \in \RV[k,\cdot]$ then $\dim_{\RV}(f(U)) < k$. Also, there is at most one elementary blowup of $(U, f)$ with respect to any coordinate of $f(U)$. We should have included the index of the ``blown up'' coordinate as a part of the data for an elementary blowup. Since, in context, either this is
clear or it does not need to be spelled out, we shall suppress mentioning it below for notational ease.

Definition~\ref{defn:blowup} is stated relative to the underlying
substructure $S$. If an object $(U, f)$ is $\lbar x$-definable with extra parameters $\lbar x$, then the iterated blowups of $(U, f)$ should be $\lbar x$-definable.

\begin{rem}\label{remark:blowup:restricted}
Let $(U^{\sharp}, f^{\sharp})$ be an elementary blowup of $(U, f) \in \RV[k,\cdot]$. Since $r_k \in \acl(\lbar r_k)$ for each $(\lbar r_k, r_k) \in f(U)$, we see that the projection map $\pr_{\leq k} : U^{\sharp} \fun U$ is an $\RV[k,\cdot]$-morphism. Also, since for each $(\lbar r_k, u) \in f^{\sharp}(U^{\sharp})$ we have
\[
(f^{\sharp})^{-1}(\lbar r_k, u)  = \bigcup_{ r_k \in A }
   (f^{-1}(\lbar r_k, r_k) \times \set{ u / r_k}),
\]
where $A = \{r_k \in \fib(f(U), \lbar r_k) : u / r_k \in \RV^{>1} \}$, clearly if $(U, f) \in \RV[k]$ then $(U^{\sharp}, f^{\sharp}) \in
\RV[k]$. So any iterated blowup of an object in $\RV[k]$ is an object in $\RV[k]$.
\end{rem}

The results below will be stated only for the more general
categories $\RV[k,\cdot]$, $\RV[*,\cdot]$, etc. But, by
Remark~\ref{remark:blowup:restricted}, they are easily seen to
hold when restricted to $\RV[k]$, $\RV[*]$, etc.\ as well.

\begin{lem}\label{elementary:blowups:preserves:iso}
Let $\mathbf{U}, \mathbf{V} \in \RV[k,\cdot]$ and $\mathbf{U}^{\sharp}$, $\mathbf{V}^{\sharp}$ two elementary blowups. If $[\mathbf{U}] = [\mathbf{V}]$ then $[\mathbf{U}^{\sharp}] = [\mathbf{V}^{\sharp}]$.
\end{lem}
\begin{proof}
Let $F : \mathbf{U} \fun \mathbf{V}$ be an isomorphism. Let $F^{\sharp} : \mathbf{U}^{\sharp} \fun \mathbf{V}^{\sharp}$ be the bijection given by $(\lbar t, s) \efun (F(\lbar t), s)$, which is easily seen to be an $\RV[k,\cdot]$-isomorphism. We leave the details to the reader.
%
\end{proof}

\begin{cor}\label{eblowup:same:loci:iso}
Let $F : \mathbf{U} \fun \mathbf{V}$ be an $\RV[k,\cdot]$-isomorphism. Let $\mathbf{U}^{\sharp}$, $\mathbf{V}^{\sharp}$ be respectively two blowups of $\mathbf{U}$, $\mathbf{V}$ with isomorphic blowup loci according to $F$. Then $\mathbf{U}^{\sharp}$, $\mathbf{V}^{\sharp}$ are isomorphic.
\end{cor}

\begin{lem}\label{isp:transitive}
Let $\mathbf{U} = (U, f)$, $\mathbf{V} = (V, g)$ be isomorphic objects in $\RV[k,\cdot]$. Let $\mathbf{U}_1$, $\mathbf{V}_1$ be two iterated blowups of $\mathbf{U}$, $\mathbf{V}$ of length $m$, $n$, respectively. Then
there are isomorphic iterated blowups $\mathbf{U}_2$, $\mathbf{V}_2$ of $\mathbf{U}_1$, $\mathbf{V}_1$ of lengths $n$, $m$, respectively.
\end{lem}
\begin{proof}
Fix an isomorphism $I: \mathbf{U} \fun \mathbf{V}$. We do induction on the sum $l = m + n$. For the base case $l = 1$, without loss of generality, we may assume that $n = 0$. Let $C$ be the blowup locus of $\mathbf{U}_1$. Clearly $\mathbf{V}$ may be blown up by using the same elementary blowup as $\mathbf{U}_1$, where the blowup locus is changed to $I(C)$. So the base case follows from Corollary~\ref{eblowup:same:loci:iso}.

We proceed to the inductive step. Let $\mathbf{U}^{\sharp}$, $\mathbf{V}^{\sharp}$ be the first blowups in $\mathbf{U}_1$, $\mathbf{V}_1$ and $C$, $D$ their blowup loci, respectively. Let $\mathbf{U}'^{\sharp}$, $\mathbf{V}'^{\sharp}$ be the corresponding elementary blowups contained in $\mathbf{U}^{\sharp}$, $\mathbf{V}^{\sharp}$. If, say, $n = 0$, then by the argument in the base case $\mathbf{V}$ may be blown up to an object that is isomorphic to $\mathbf{U}^{\sharp}$ and hence the inductive
hypothesis may be applied. So let us assume that $m,n > 0$. Let $A = C \cap I^{-1}(D)$ and $B = I(C) \cap D$. Since $(A, f \rest A)$ and $(B, g \rest B)$ are isomorphic, by Lemma~\ref{elementary:blowups:preserves:iso}, the
elementary sub-blowups of $\mathbf{U}'^{\sharp}$, $\mathbf{V}'^{\sharp}$ that correspond to $(A, f \rest A)$ and $(B, g \rest B)$ are isomorphic. Then, it is not hard to see that the blowup $\mathbf{U}^{\sharp\sharp}$ of $\mathbf{U}^{\sharp}$ using the locus $I^{-1}(D) \mi C$ and its corresponding elementary sub-blowup of $\mathbf{V}'^{\sharp}$ and the blowup $\mathbf{V}^{\sharp\sharp}$ of $\mathbf{V}^{\sharp}$ using the locus $I(C) \mi D$ and its corresponding elementary sub-blowup of $\mathbf{U}'^{\sharp}$ are isomorphic.

Applying the inductive hypothesis to the iterated blowups $\mathbf{U}^{\sharp\sharp}$, $\mathbf{U}_1$ of $\mathbf{U}^{\sharp}$, we obtain an iterated blowup $\mathbf{U}^{\sharp3}$ of $\mathbf{U}^{\sharp\sharp}$ of length $m - 1$ and an iterated blowup $\mathbf{U}_1^{\sharp}$ of $\mathbf{U}_1$ of length $1$ such that they are isomorphic. Similarly, we obtain an iterated blowup $\mathbf{V}^{\sharp3}$ of $\mathbf{V}^{\sharp\sharp}$ of length $n - 1$ and an iterated blowup $\mathbf{V}_1^{\sharp}$ of $\mathbf{V}_1$ of length $1$ such that they are isomorphic. Applying the inductive hypothesis again to the iterated blowups $\mathbf{U}^{\sharp3}$, $\mathbf{V}^{\sharp3}$ of $\mathbf{U}^{\sharp\sharp}$, $\mathbf{V}^{\sharp\sharp}$, we obtain an iterated blowup $\mathbf{U}^{\sharp4}$ of $\mathbf{U}^{\sharp3}$ of length $n - 1$ and an iterated blowup $\mathbf{V}^{\sharp4}$ of $\mathbf{V}^{\sharp3}$ of length $m - 1$ such that they are isomorphic. Finally, applying the inductive hypothesis to the iterated blowups
$\mathbf{U}^{\sharp4}$, $\mathbf{U}_1^{\sharp}$ of $\mathbf{U}^{\sharp3}$, $\mathbf{U}_1^{\sharp}$ and the iterated blowups $\mathbf{V}^{\sharp4}$, $\mathbf{V}_1^{\sharp}$ of $\mathbf{V}^{\sharp3}$, $\mathbf{V}_1^{\sharp}$, we obtain an iterated blowup $\mathbf{U}_2$ of $\mathbf{U}_1^{\sharp}$ of length $n - 1$ and an iterated blowup $\mathbf{V}_2$ of $\mathbf{V}_1^{\sharp}$ of length $m - 1$ such that $\mathbf{U}^{\sharp4}$, $\mathbf{U}_2$, $\mathbf{V}^{\sharp4}$, and $\mathbf{V}_2$ are all isomorphic. This
process is illustrated as follows:
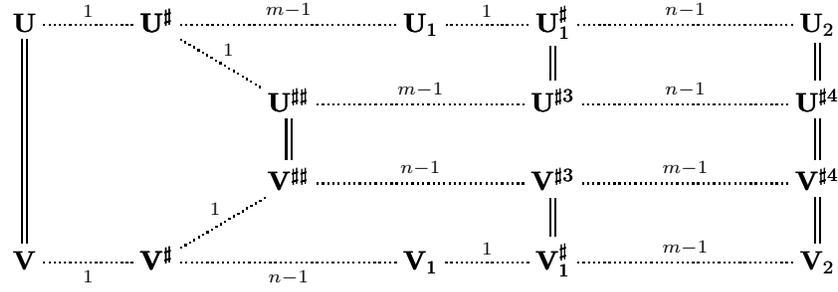
\[
\bfig
  \square(0,0)/.`=``./<500,900>[\mathbf{U}`\mathbf{U}^{\sharp}`
  \mathbf{V}`\mathbf{V}^{\sharp};1```1]
  \square(500,0)/.```./<1000,900>[\mathbf{U}^{\sharp}`\mathbf{U}_1`
  \mathbf{V}^{\sharp}`\mathbf{V}_1; m - 1```n - 1]
  \morphism(500,900)/./<500,-300>[\mathbf{U}^{\sharp}`\mathbf{U}^{\sharp\sharp};1]
  \morphism(500,0)/./<500,300>[\mathbf{V}^{\sharp}`
   \mathbf{V}^{\sharp\sharp};1]
  \morphism(1000,300)/=/<0,300>[\mathbf{V}^{\sharp\sharp}`
   \mathbf{U}^{\sharp\sharp};]
  \morphism(1500,900)/./<500,0>[\mathbf{U}_1`\mathbf{U}_1^{\sharp};1]
  \morphism(1500,0)/./<500,0>[\mathbf{V}_1`\mathbf{V}_1^{\sharp};1]
  \morphism(1000,600)/./<1000,0>[\mathbf{U}^{\sharp\sharp}`\mathbf{U}^{\sharp3};
   m - 1]
  \morphism(1000,300)/./<1000,0>[\mathbf{V}^{\sharp\sharp}`\mathbf{V}^{\sharp3};
   n - 1]
  \morphism(2000,900)/=/<0,-300>[\mathbf{U}_1^{\sharp}`\mathbf{U}^{\sharp3};]
  \morphism(2000,0)/=/<0,300>[\mathbf{V}_1^{\sharp}`\mathbf{V}^{\sharp3};]
  \morphism(2000,600)/./<1000,0>[\mathbf{U}^{\sharp3}`\mathbf{U}^{\sharp4};n - 1]
  \morphism(2000,300)/./<1000,0>[\mathbf{V}^{\sharp3}`\mathbf{V}^{\sharp4};m - 1]
  \morphism(3000,300)/=/<0,300>[\mathbf{V}^{\sharp4}`\mathbf{U}^{\sharp4};]
  \morphism(2000,900)/./<1000,0>[\mathbf{U}_1^{\sharp}`\mathbf{U}_2;n - 1]
  \morphism(2000,0)/./<1000,0>[\mathbf{V}_1^{\sharp}`\mathbf{V}_2;m - 1]
  \morphism(3000,0)/=/<0,300>[\mathbf{V}_2`\mathbf{V}^{\sharp4};]
  \morphism(3000,900)/=/<0,-300>[\mathbf{U}_2`\mathbf{U}^{\sharp4};]
\efig
\]
So $\mathbf{U}_2$ and $\mathbf{V}_2$ are as desired.
\end{proof}

\begin{cor}\label{blowup:equi:class}
Let $[\mathbf{U}] = [\mathbf{U}']$ and $[\mathbf{V}] = [\mathbf{V}']$ in $\RV[k,\cdot]$. Suppose that there are isomorphic iterated blowups of $\mathbf{U}$, $\mathbf{V}$. Then there are isomorphic iterated blowups of $\mathbf{U}'$, $\mathbf{V}'$.
\end{cor}

\begin{defn}
Let $\isp[k, \cdot]$ be the subclass of $\ob \RV[k,\cdot] \times \ob \RV[k,\cdot]$ of those pairs $(\mathbf{U}, \mathbf{V})$ such that there exist isomorphic iterated blowups $\mathbf{U}^{\sharp}$, $\mathbf{V}^{\sharp}$. Let
\begin{gather*}
\isp[*, \cdot] = \coprod_{0 \leq i} \isp[i, \cdot], \quad \isp[k] = \isp[k, \cdot] \cap (\ob \RV[k] \times \ob \RV[k]),\\
\isp[*] = \isp[*, \cdot] \cap \coprod_{0 \leq i}(\ob \RV[i] \times \ob \RV[i]).
\end{gather*}
\end{defn}

We will just write $\isp$ for all these classes if there is no
danger of confusion. When the underlying substructure $S$ is
expanded with some extra parameters $\lbar x$ we shall write $\isp
\la \lbar x \ra$ for the accordingly expanded classes.

By Corollary~\ref{blowup:equi:class}, $\isp$ may be regarded as a
binary relation on isomorphism classes.

\begin{lem}\label{isp:congruence}
$\isp[k, \cdot]$ is a semigroup congruence relation and $\isp[*,
\cdot]$ is a semiring congruence relation.
\end{lem}
\begin{proof}
Obviously $\isp[k, \cdot]$ is reflexive and symmetric. Suppose that $([\mathbf{U}_1], [\mathbf{U}_2]), ([\mathbf{U}_2], [\mathbf{U}_3]) \in
\isp[k, \cdot]$. Then, by Lemma~\ref{isp:transitive}, there are
iterated blowups $\mathbf{U}_1^{\sharp}$ of $\mathbf{U}_1$, $\mathbf{U}_{2}^{\sharp 1}$ and $\mathbf{U}_{2}^{\sharp 2}$ of $\mathbf{U}_2$, and
$\mathbf{U}_3^{\sharp}$ of $\mathbf{U}_3$ such that they are all isomorphic. So $\isp[k, \cdot]$ is transitive and hence is an
equivalence relation. For any $[\mathbf{W}] \in \gsk \RV[l,\cdot]$, the following are
easily checked:
\[
([\mathbf{U}_1 \uplus \mathbf{W}], [\mathbf{U}_2 \uplus \mathbf{W}])\in \isp[k, \cdot],\quad
([\mathbf{U}_1 \times \mathbf{W}], [\mathbf{U}_2 \times \mathbf{W}])\in \isp[*, \cdot].
\]
These yield the desired congruence relations.
\end{proof}


\subsection{Blowups and special bijections}

For any $(U, f) \in \RV[k,\cdot]$ and any special bijection $T$ on $\bb L(U, f)$, we shall write $U_T$ for the subset $(\prv \circ T)(\mathbb{L}(U, f))$ and $\mathbf{U}_T$ for the object $(U_T, \pr_{\leq k})$.

\begin{lem}\label{special:bi:lh:1:equi:blowup}
Let $(U, f) \in \RV[k,\cdot]$ and $\eta$ a centripetal transformation on $\mathbb{L}(U, f)$ with focus map $\lambda$ such that the locus of $\lambda$ is $\mathbb{L}(U, f)$. Then $\mathbf{U}_{\can \circ \eta}$ is isomorphic to an elementary blowup of $(U, f)$.
\end{lem}
\begin{proof}
Without loss of generality, we may assume $\dom(\lambda) = \pr_{> 1} (\mathbb{L}(U, f))$ and $0 \notin \ran(\lambda)$, that is, $\infty \notin \pr_1 (f(U))$. Since $\lambda$ is a function, for every $(r_1, \lbar r_1) \in f(U)$ and every $\lbar a_1 \in \rv^{-1}(\lbar r_1)$ we have, by Lemma~\ref{dcl:to:ac}, $r_1 \in \acl(\lbar a_1)$ and hence, by Lemma~\ref{acl:VF:transfer:acl:RV}, $r_1 \in \acl(\lbar
r_1)$. So the elementary blowup $(U^{\sharp}, f^{\sharp})$ of $(U,
f)$ in the first coordinate of $f(U)$ does exist. It is easy to see that the function $F : U_{\can \circ \eta} \fun U^{\sharp}$ given by $(r_1, \lbar r_1, f(\lbar t), \lbar t) \efun (\lbar t, r_1/f_{1}(\lbar t))$ is an isomorphism between $\mathbf{U}_{\can \circ \eta}$ and $(U^{\sharp}, f^{\sharp})$, where $\lbar t \in U$ and $f(\lbar t) = (f_{1}(\lbar t), \lbar r_1)$.
\end{proof}

\begin{cor}\label{special:correspond:blowup}
Let $\mathbf{U} \in \RV[k,\cdot]$ and $T$ a special bijection on $\bb L \mathbf{U}$. Then $\mathbf{U}_T$ is isomorphic to an iterated blowup of $\mathbf{U}$.
\end{cor}
\begin{proof}
By induction on the length $\lh(T)$ of $T$ and Lemma~\ref{isp:transitive}, this is easily reduced to the case $\lh(T) = 1$, which follows from
Lemma~\ref{special:bi:lh:1:equi:blowup}.
\end{proof}

\begin{cor}\label{kernel:dim:1}
Let $\mathbf{U}_1, \mathbf{U}_2 \in \RV[1,\cdot]$. If $\bb L \mathbf{U}_1$ is definably bijective to $\bb L\mathbf{U}_2$ then $([\mathbf{U}_1], [\mathbf{U}_2]) \in \isp$.
\end{cor}
\begin{proof}
By Remark~\ref{special:dim:1:RV:iso}, there are special bijections $T_1$, $T_2$ on $\bb L \mathbf{U}_1$, $\bb L \mathbf{U}_2$ such that $\mathbf{U}_{T_1}$, $\mathbf{U}_{T_2}$ are isomorphic. So the assertion follows from Corollary~\ref{special:correspond:blowup}.
\end{proof}

\begin{lem}\label{elementary:blowup:same:VF}
Suppose that the substructure $S$ is $(\VF, \Gamma)$-generated. Let $(U^{\sharp}, f^{\sharp})$ be an elementary blowup of $(U, f) \in \RV[k,\cdot]$. Then there is a special bijection $T$ of length $1$ on $\mathbb L(U, f)$ such that the locus of $T$ is $\mathbb L(U, f)$ and there is a commutative diagram
\[
\bfig
  \square(700,0)/->`->`->`->/<700,400>[T(\mathbb{L}(U, f))`
  U_T`\mathbb{L}(U^{\sharp}, f^{\sharp})`U^{\sharp};
  \prv`F`F_{\downarrow}`\prv]
  \morphism(0,400)/->/<700,0>[\mathbb{L}(U, f)`T(\mathbb{L}(U, f));T]
 \efig
\]
where both $F$ and
$F_{\downarrow}$ are definable bijections. Moreover, $F$ is given by
\[
(\lbar a, b, \rv(b), t, \lbar s) \efun (\lbar a, b, \rv(b)/t, \lbar s),
\]
that is, $F$ is unary relative to all $\VF$-coordinates.
\end{lem}
\begin{proof}
Suppose that $\lbar t = (\lbar t_k, t_k)$ is the only element in $f(U)$. For any centripetal transformation $\eta$ on
$\rv^{-1}(\lbar t)$ with respect to a focus map $\lambda$ on
$\rv^{-1}(\lbar t_k)$, the function
\[
F_{\lbar t} : (\can \circ \eta)(\rv^{-1}(\lbar t) \times U) \fun
\bb L(U^{\sharp}, f^{\sharp})
\]
given by
\[
(\can \circ \eta)(\lbar a_k, a_k, \lbar s) \efun (\lbar a_k, a_k -
\lambda(\lbar a_k), \lbar s, \rv(a_k - \lambda(\lbar a_k))/t_k )
\]
is a bijection as required. So, by compactness, it is enough to show that there is a $\lbar t$-definable focus map $\lambda$ such that $\rv^{-1}(\lbar t_k) \times U \sub \dom(\lambda)$. Let $\vrv(\lbar t) = (\lbar \gamma_n, \gamma_n) = \lbar \gamma$. Since $t_k \in \acl(\lbar t_k)$ and $t_k \neq \infty$, by Lemma~\ref{exists:gamma:polynomial}, there is a
$\lbar \gamma$-polynomial $P(\lbar X)$ with coefficients in $\VF(S)$ such that $\lbar t$ is a residue root of $P(\lbar X)$ but is not a residue root of $\partial
P(\lbar X) / \partial X_k$. This means that, for every $\lbar a_k \in \rv^{-1}(\lbar t_k)$, $t_k$ is a simple residue root of the $\gamma_n$-polynomial
$P(\lbar a_k, X_k)$ and hence, by Lemma~\ref{hensel:lemma}, there is a unique $a_k \in \rv^{-1}(t_k)$ such that $P(\lbar a_k, a_k) = 0$. So there exists
a focus map as desired.
\end{proof}



\begin{cor}\label{iterated:blowup:same:RV}
Suppose that the substructure $S$ is $(\VF, \Gamma)$-generated. Let $\mathbf{U}^{\sharp}$ be an iterated blowup of $\mathbf{U}$ of length $l$. Then $\bb L \mathbf{U}$ and $\bb L \mathbf{U}^{\sharp}$ are definably bijective.
\end{cor}
\begin{proof}
By induction this is immediately reduced to the case $l=1$, which follows from Lemma~\ref{elementary:blowup:same:VF} and Theorem~\ref{L:semigroup:hom}.
\end{proof}

\begin{lem}\label{isp:VF:fiberwise:contract:isp}
Let $A_1 \sub \VF^n \times \RV^{m_1}, A_2 \sub \VF^n \times \RV^{m_2}$ be two definable subsets such that $\pvf (A_1) = \pvf (A_2) = A$. Suppose that there is an $E \sub \mathds{N}$ such that, for every $\lbar a \in A$,
\[
([\fib(A_1, \lbar a)]_{E}, [\fib(A_2, \lbar a)]_{E}) \in \isp \la \lbar a \ra.
\]
Let $\wh T_{\sigma}$, $\wh R_{\sigma}$ be two standard contractions of $A_1$, $A_2$. Set $E' = E \cup I_n$. Then
\[
([\wh T_{\sigma}(A_1)]_{E'}, [\wh R_{\sigma}(A_2)]_{E'}) \in \isp.
\]
\end{lem}
\begin{proof}
By induction on $n$ this is immediately reduced to the case $n=1$. So let us assume $A \sub \VF$. By an argument similar to the one in the proof of Lemma~\ref{bijection:made:contractible}, there is a special
bijection $T_A$ on $A$ such that the following items hold:
\begin{enumerate}
 \item $T_A(A) = A^{\sharp}$ is an $\RV$-pullback.


 \item Let $h_1 = T_A \circ (\pvf \rest A_1)$ and $h_2 = T_A \circ (\pvf \rest A_2)$. For any $\rv$-polydisc $\gp = \rv^{-1}(t) \times \set{(t, \lbar s)} \sub A^{\sharp}$,
     \[
     h_1^{-1}(\gp) = A_{\gp} \times U_{\gp,1}, \quad h_2^{-1}(\gp) = A_{\gp} \times U_{\gp,2},
     \]
      where $A_{\gp} \sub A$ and, for any $a \in A_{\gp}$, $U_{\gp,1} = \fib(A_1, a)$ and $U_{\gp,2} = \fib(A_2, a)$.

 \item There is a formula $\phi$ such that, for any $\gp = \rv^{-1}(t) \times \set{(t, \lbar s)} \sub A^{\sharp}$ and any $a \in A_{\gp}$, $\phi(a)$ defines the same iterated blowups that witness $([U_{\gp,1}]_E, [U_{\gp,2}]_E) \in \isp \la a \ra$. Clearly these iterated blowups are also $(t, \lbar s)$-definable. Therefore, $([U_{\gp,1}]_E, [U_{\gp,2}]_E) \in \isp \la t, \lbar s\ra$.
\end{enumerate}
Now let
\[
A_1^{\sharp} = \bigcup_{a \in A} (\set{T_A(a)} \times \fib(A_1, a)), \quad A_2^{\sharp} = \bigcup_{a \in A} (\set{T_A(a)} \times \fib(A_2, a)).
\]
Note that $T_A(\fib(A_1, \lbar t))$ is an $\RV$-pullback for every $\lbar t \in \prv(A_1)$, similarly for $A_2$. So $A_1^{\sharp}, A_2^{\sharp}$ may be obtained by special bijections on $A_1$, $A_2$.

By the second item above, for any $\lbar t \in \pr_E (A_1)$, $\fib(A_1^{\sharp}, \lbar t)$ is an $\RV$-pullback that is $\lbar t$-definably bijective to $\fib(A_1, \lbar t)$ and hence to the $\lbar t$-definable $\RV$-pullback $\fib(T_{\sigma}(A_1), \lbar t)$. By Corollary~\ref{kernel:dim:1}, we have
\[
([\prv\fib(A_1^{\sharp}, \lbar t)]_1, [\fib(\wh T_{\sigma}(A_1), \lbar t)]_1) \in \isp \la \lbar t \ra
\]
and hence, by compactness,
\[
([\prv (A_1^{\sharp})]_{E'}, [\wh T_{\sigma}(A_1)]_{E'}) \in \isp.
\]
Symmetrically we have
\[
([\prv (A_2^{\sharp})]_{E'}, [\wh R_{\sigma}(A_2)]_{E'}) \in \isp.
\]
On the other hand, for any $\rv$-polydisc $\gp = \rv^{-1}(t) \times \set{(t, \lbar s)} \sub A^{\sharp}$, we have $\gp \times U_{\gp,1} \sub A_1^{\sharp}$ and $\gp \times U_{\gp,2} \sub A_2^{\sharp}$, and hence, by the third item above,
\[
([\prv(\gp \times U_{\gp,1})]_E, [\prv(\gp \times U_{\gp,2})]_E) \in \isp \la t, \lbar s \ra.
\]
Observe that, by Lemma~\ref{function:rv:to:vf:finite:image}, $\lbar s \in \acl(a)$ for every $a \in \rv^{-1}(t)$ and hence, by Lemma~\ref{acl:VF:transfer:acl:RV}, $\lbar s \in \acl(t)$. So, by compactness, we deduce
\[
([\prv (A_1^{\sharp})]_{E'}, [\prv (A_2^{\sharp})]_{E'}) \in \isp.
\]
Since $\isp$ is a congruence relation, the lemma follows.
\end{proof}

\begin{cor}\label{contraction:same:perm:isp}
Let $A_1 \sub \VF^n \times \RV^{m_1}$ and $A_2 \sub \VF^n \times
\RV^{m_2}$ be two definable subsets and $f : A_1 \fun A_2$ a unary
bijection relative to the coordinate $i$. Then for any permutation
$\sigma$ of $I_n$ with $\sigma(1) = i$ and any standard contractions $\wh T_{\sigma}, \wh R_{\sigma}$ of $A_1, A_2$,
\[
([\wh T_{\sigma}(A_1)]_{\leq n}, [\wh R_{\sigma}(A_2)]_{\leq n}) \in \isp.
\]
\end{cor}
\begin{proof}
For any $\lbar a \in \pr_{\widetilde{i}} (A_1) = \pr_{\widetilde{i}} (A_2)$ and any $\lbar a$-definable standard contractions $\wh T$, $\wh R$ of $\fib(A_1, \lbar a)$, $\fib(A_2, \lbar a)$, by Corollary~\ref{kernel:dim:1}, we have
\[
([\wh T(\fib(A_1, \lbar a))]_1, [\wh R(\fib(A_2, \lbar a))]_1) \in
\isp \la \lbar a \ra.
\]
Then the corollary follows from
Lemma~\ref{isp:VF:fiberwise:contract:isp}.
\end{proof}

The following lemma is essentially a version of Fubini's theorem.

\begin{lem}\label{contraction:perm:pair:isp}
Let $A \sub \VF^n \times \RV^{m}$ be a definable subset. Let $i, j \in I_n$ be distinct and $\sigma_1$, $\sigma_2$ two permutations of $I_n$ such that
\[
\sigma_1(1) = \sigma_2(2) = i, \quad \sigma_1(2) = \sigma_2(1) = j, \quad \sigma_1
\rest \set{3, \ldots, n} = \sigma_2 \rest \set{3, \ldots, n}.
\]
Then, for any standard contractions $\wh T_{\sigma_1}$, $\wh T_{\sigma_2}$ of $A$,
\[
([\wh T_{\sigma_1}(A)]_{\leq n}, [\wh T_{\sigma_2}(A)]_{\leq n}) \in \isp.
\]
\end{lem}
\begin{proof}
Let $ij$, $ji$ denote the permutations of $\{i, j\}$. By compactness and
Lemma~\ref{isp:VF:fiberwise:contract:isp}, it is enough to show
that, for any $\lbar a \in \pr_{\widetilde{\{i, j\}}}(A)$ and any standard
contractions $\wh T_{ij}$, $\wh T_{ji}$ of $\fib(A, \lbar a)$,
\[
([\wh T_{ij}(\fib(A, \lbar a))]_{\leq 2}, [\wh T_{ji}(\fib(A, \lbar a))]_{\leq 2})
\in \isp \la \lbar a \ra.
\]
To that end, fix an $\lbar a \in \pr_{\widetilde{\{i, j\}}}(A)$ and let $B = \fib(A, \lbar a)$. By Corollary~\ref{subset:partitioned:2:unit:contracted}, there are a definable bijection $f : B \fun \VF^2 \times \RV^l$ that is unary relative to both coordinates and two standard contractions $\wh R_{ij}$, $\wh R_{ji}$ of $f(B)$ such that
\[
[\wh R_{ij}(f(B))]_{\leq 2} = [\wh R_{ji}(f(B))]_{\leq 2}
\]
in the corresponding $\RV$-category with respect to $\seq{\lbar a}$. So the desired property follows from Corollary~\ref{contraction:same:perm:isp}.
\end{proof}

If the substructure $S$ is $(\VF, \Gamma)$-generated then the
congruence relation $\isp$ is the congruence relation induced by
$\bb L$:

\begin{prop}\label{kernel:L}
Suppose that the substructure $S$ is $(\VF, \Gamma)$-generated. Let $\mathbf{U}, \mathbf{V} \in \RV[k,\cdot]$. Then
\[
[\bb L\mathbf{U}] = [\bb L \mathbf{V}] \quad \text{if and only if} \quad ([\mathbf{U}], [\mathbf{V}]) \in \isp.
\]
\end{prop}
\begin{proof}
For the ``only if'' direction, suppose that $F : \bb L \mathbf{U} \fun \bb L \mathbf{V}$ is a definable bijection. By Lemma~\ref{bijection:partitioned:unary}, there is a partition of $\bb L \mathbf{U}$ into definable subsets $A_1, \ldots, A_n$  such that each $F_i = F \rest A_i$ is a composition of relatively unary bijections. Applying Theorem~\ref{special:bi:polynomial:constant} as in Lemma~\ref{bijection:made:contractible}, we obtain special bijections
$T$, $R$ on $\bb L \mathbf{U}$, $\bb L \mathbf{V}$ such that $T(A_i)$, $(R \circ F)(A_i)$ are $\RV$-pullbacks for each $i$. By Corollary~\ref{special:correspond:blowup}, it is enough to show that there are standard contractions $\wh T_{\sigma}$, $\wh R_{\tau}$ of $T(A_i)$, $(R \circ F)(A_i)$ for each $i$ such that
\[
([(\wh T_{\sigma} \circ T)(A_i)]_{\leq k}, [(\wh R_{\tau} \circ R \circ F)(A_i)]_{\leq k}) \in \isp.
\]
To that end, first note that each $(R \circ F \circ T^{-1}) \rest T(A_i)$ is a composition of relatively unary bijections, say
\[
T(A_i) = B_1 \to^{G_1} B_2 \cdots B_l \to^{G_l} (R \circ F)(A_i) = B_{l+1}.
\]
For each $j \leq l - 1$ we may apply Corollary~\ref{contraction:same:perm:isp} to $B_j \to^{G_j} B_{j+1}$ and $B_{j+1} \to^{G_{j+1}} B_{j+2}$ in the obvious way to get two pairs of $\isp$-congruent contractions $([U_j]_{\leq k}, [U_{j+1}]_{\leq k})$ and $([U'_{j+1}]_{\leq k}, [U_{j+2}]_{\leq k})$ such that the two permutations of $I_k$ used on $B_{j+1}$, if distinct, are as in Lemma~\ref{contraction:perm:pair:isp}. Then, by Corollary~\ref{contraction:same:perm:isp} and Lemma~\ref{contraction:perm:pair:isp}, we have
\[
([U_{j+1}]_{\leq k}, [U'_{j+1}]_{\leq k}) \in \isp.
\]
This completes the ``only if'' direction.

The ``if'' direction follows from Corollary~\ref{iterated:blowup:same:RV} and
Theorem~\ref{L:semigroup:hom}.
\end{proof}

\subsection{Invariants of definable bijections}\label{subsection:iso:inv}

In this subsection we shall assume that the substructure $S$ is $(\VF, \Gamma)$-generated.

As above, the results will be stated for the more general categories $\RV[k,\cdot]$, $\RV[*,\cdot]$, etc. By Remark~\ref{remark:blowup:restricted}, it is not hard to see that analogous results may be derived for the restricted categories $\RV[k]$, $\RV[*]$, etc.\ if the arguments are accordingly restricted.

\begin{prop}\label{main:prop:k}
For each $k \geq 0$ there is a canonical isomorphism of Grothendieck semigroups
\[
\Xint\infty_{+} : \gsk \VF[k,\cdot] \fun \gsk \RV[k,\cdot]/\isp
\]
such that
\[
\Xint\infty_{+} [A] = [\mathbf{U}]/\isp \quad \text{if and only if} \quad [A] = [\bb L \mathbf{U}].
\]
\end{prop}
\begin{proof}
By Theorem~\ref{L:semigroup:hom}, $\bb L$ induces a canonical semigroup homomorphism
\[
\bb L : \gsk \RV[k,\cdot] \fun \gsk \VF[k, \cdot].
\]
By Theorem~\ref{L:surjective}, $\bb L$ is surjective. By
Proposition~\ref{kernel:L}, the semigroup congruence relation
induced by $\bb L$ is precisely $\isp$ and hence $\gsk
\RV[k,\cdot]/\isp$ is canonically isomorphic to $\gsk
\VF[k,\cdot]$.
\end{proof}

For each $k > 0$ let $\gsk \RV^{\times}[k,\cdot]$ be the sub-semigroup of $\gsk \RV[k,\cdot]$ that contains $[\0]_k$ and those elements $[(U,f)]$ with $f(U) \sub (\RV^{\times})^k$. For $k =0$ let $\gsk \RV^{\times}[0,\cdot] = \gsk \RV[0,\cdot]$. We have the direct sums:
\[
\gsk \RV^{\times}[\leq k,\cdot] =  \bigoplus_{i \leq k} \gsk
\RV^{\times}[i,\cdot] \sub  \bigoplus_{0 \leq k}
\gsk \RV^{\times}[k,\cdot]  = \gsk \RV^{\times}[*,\cdot].
\]
Recall the maps $\bb E_{i, j}$ from~\cite[Definition~4.17]{Yin:special:trans}. For each $k \geq 0$, let $\bb F_k$ be the obviously surjective semigroup homomorphism
\[
\bigoplus_{i \leq k} \bb E_{i, k} : \gsk \RV^{\times}[\leq k,\cdot] \fun \gsk \RV[k,\cdot].
\]
It is also clear from the condition on weight in~\cite[Definition~4.16]{Yin:special:trans} that $\bb F_k$ is injective as well. For every $k \geq 0$ we have a commutative diagram:
\[
\bfig
  \Square(0,0)/^{ (}->`->`->`->/<400>[\gsk \RV^{\times}{[\leq k, \cdot]}
  `\gsk \RV^{\times}{[\leq k + 1, \cdot]}`\gsk \RV{[k,\cdot]}`\gsk \RV{[k+1,\cdot]};
   `\bb F_k`\bb F_{k+1}`\bb E_{k}]
 \efig
\]
Let $\ispt[\leq k, \cdot]$ be the semigroup congruence relation on $\gsk \RV^{\times}[\leq k,\cdot]$ induced by $\bb F_k$ and $\isp$. It is easy to see that $\ispt[\leq k, \cdot]$ is the restriction of $\ispt[\leq k+1, \cdot]$ to $\gsk \RV^{\times}[\leq k,\cdot]$. So
\[
\ispt[*,\cdot] = \bigcup_{0 \leq k} \ispt[\leq k, \cdot]
\]
is a semiring congruence relation on $\gsk \RV^{\times}[*,\cdot]$. As above, for notational ease, all these congruence relations shall be simply denoted by $\ispt$. For every $k \geq 0$, Proposition~\ref{main:prop:k} induces a commutative diagram:
\[
\bfig
  \Square(0,0)/^{ (}->`->`->`^{ (}->/<400>[\gsk \VF{[k,\cdot]}`\gsk \VF{[k+1,\cdot]}`\gsk \RV^{\times}{[\leq k,\cdot]}/\ispt`\gsk \RV^{\times}{[\leq k+1,\cdot]}/\ispt;`\int_{+}`\int_{+}`]
 \efig
\]
Putting these together we obtain:

\begin{thm}\label{theorem:semiring}
There is a canonical isomorphism of Grothendieck semirings
\[
\int_{+} : \gsk \VF_*[\cdot] \fun \gsk \RV^{\times}[*,\cdot]/\ispt
\]
such that
\[
\int_{+} [A] = [\mathbf{U}]/\ispt \quad \text{if and only if} \quad [A] = [\bb L \mathbf{U}].
\]
\end{thm}

\begin{rem}\label{remark:fn:semimodule}
Let us write $\gsk \VF_*[\cdot]$, $\gsk \RV^{\times}[*,\cdot]$ simply as $\gsk \VF$, $\gsk \RV$ respectively. Let $A \sub \VF^n$ and $f : A \fun \mdl P(\RV^m)$ a definable function. Each $f(\lbar a)$ may be treated as an object in $\RV^{\times}[*,\cdot]$ with respect to some coordinate projection and hence $f$ is a representative of an equivalence class of definable functions. Such an equivalence class is understood as a \emph{definable function}
\[
A \fun \gsk \RV /\ispt \quad \text{ given by } \quad \lbar a \efun [f(\lbar a)]/\ispt \la \lbar a \ra.
\]
which, for simplicity, is also denoted by $f$. This should not cause confusion in context. Note: here the target $\gsk \RV/\ispt$ is really a sequence of semirings that vary according to the substructures $\dcl(\lbar a)$. Only with this understanding are we able to accommodate the two values $f(\lbar a_1)$, $f(\lbar a_2)$ in one semiring, namely the semiring with respect to the substructure $\dcl(\lbar a_1, \lbar a_2)$. As usual, for simplicity, the extra parameters may not always be spelled out in context.

Let $\bb L f = \bigcup_{\lbar a \in A}(\{\lbar a\} \times \bb L(f(\lbar a)))$. We define $\int_{+A} f$ to be $\int_+ [f] = \int_+ [\bb L f]$, which, by Proposition~\ref{kernel:L} and compactness, does not depend on the choice of the representative $f$. Let $\fn(A, \gsk \RV/\ispt)$ be the set of all definable functions on $A$, which is a natural $\gsk \RV/\ispt$-semimodule. Then we have a homomorphism of $\gsk \RV/\ispt$-semimodules:
\[
\int_{+A} : \fn(A, \gsk \RV/\ispt) \fun \gsk \RV/\ispt.
\]
\end{rem}

Let $E \sub I_n$ be a nonempty subset. For any $ f \in \fn(A, \gsk \RV/\ispt)$ let $\pr_E  f$ be the function on $\pr_E(A)$ given by
\[
\lbar a \efun \int_{+\fib(A, \lbar a)} f  = \int_{+} [f \rest \fib(A, \lbar a)].
\]
By compactness, $\pr_E f \in \fn(\pr_E(A), \gsk \RV/\ispt)$. We also write
\[
\int_{+\lbar a \in \pr_E(A)} \int_{+ \fib(A, \lbar a)} f = \int_{+\pr_E(A)} \pr_E f.
\]

\begin{prop}
For any nonempty subsets $E_1, E_2 \sub I_n$,
\[
\int_{+\lbar a \in \pr_{E_1}(A)} \int_{+ \fib(A, \lbar a)} f = \int_{+\lbar a \in \pr_{E_2}(A)} \int_{+ \fib(A, \lbar a)} f.
\]
\end{prop}
\begin{proof}
Since both sides may be represented by a standard contraction of $f$, this is immediate by Proposition~\ref{kernel:L}.
\end{proof}

In the groupification $\ggk \RV / \bf I$ of $\gsk \RV /\ispt$, $\ispt$ determines uniquely an ideal $\bf I$. We shall give a simple algebraic description of this ideal as follows. First observe that
\[
([1]_1, [1]_0 \oplus [(\RV^{\times})^{>1}]_1) \in \ispt,
\]
where $(\RV^{\times})^{>1} = \RV^{>1} \mi \{\infty\}$. Let $[(U, f)] \in \gsk \RV^{\times}[k,\cdot]$ and $(U^{\sharp}, f^{\sharp})$ an elementary blowup of $(U, f)$. Let
\[
U' = U \times \{\infty\}, \quad U'' = U^{\sharp} \mi U', \quad f' = \bb E_{k-1}^{-1}(f^{\sharp} \rest U'), \quad f'' = f^{\sharp} \rest U''.
\]
We clearly have
\[
[(U', f')] \times  [1]_1 = [(U, f)], \quad
[(U', f')] \times [(\RV^{\times})^{>1}]_1 = [(U'', f'')].
\]
Hence
\begin{align*}
[(U', f')] \oplus [(U'', f'')] &= [(U', f')] \oplus ([(U', f')] \times [(\RV^{\times})^{>1}]_1)\\
              &= [(U', f')] \times ([1]_0 \oplus [(\RV^{\times})^{>1}]_1)\\
              &=_{\ispt} [(U', f')] \times [1]_1\\
              &= [(U, f)].
\end{align*}
This shows that, as a semiring congruence relation on $\gsk \RV$, $\ispt$ is generated by $([1]_1, [1]_0
\oplus [(\RV^{\times})^{>1}]_1)$ and hence its corresponding ideal
$\bf I$ in the graded ring $\ggk \RV$ is
generated by the element $[1]_0 \oplus ([(\RV^{\times})^{>1}]_1 -
[1]_1)$. Let
\[
\mathbf{j} = [1]_1 - [(\RV^{\times})^{>1}]_1.
\]
We now compute in $\ggk \RV$:
\[
[\mathbf{U}] = [\mathbf{U}] \times [1]_0  =_{\bf I} [\mathbf{U}] \times [1]_1 - [\mathbf{U}] \times [(\RV^{\times})^{>1}]_1 = [\mathbf{U}] \times \mathbf{j}.
\]
Iterating this computation we see that
\[
\ggk \RV/\mathbf I \cong \colim{k} \ggk \RV^{\times}[k,\cdot],
\]
where the maps of the colimit system are given by $[\mathbf{U}] \efun [\mathbf{U}] \times \mathbf j$. Consequently, the groupification of the isomorphism $\int_{+}$ may be understood as a ring isomorphism
\[
\int : \ggk \VF \fun \colim{k} \ggk \RV^{\times}[k,\cdot].
\]
Since this colimit may be embedded into the zeroth graded piece of the graded ring $\ggk \RV[\mathbf{j}^{-1}]$ via the map determined by
\[
[\mathbf{U}] \efun [\mathbf{U}] \times \mathbf{j}^{-k}
\]
for $[\mathbf{U}] \in \ggk \RV^{\times}[k,\cdot]$, we have the following:

\begin{thm}\label{theorem:integration:1}
The Grothendieck semiring isomorphism $\int_{+}$ induces canonically an injective ring homomorphism
\[
\int : \ggk \VF \fun \ggk \RV[\mathbf{j}^{-1}].
\]
\end{thm}

\section{The kernel of $\mathbb{L}$ with volume forms and integration}\label{section:integrals}

Throughout this section we shall assume that the substructure $S$ is $(\VF, \Gamma)$-generated. We shall only discuss the categories $\mRV[k]$ and $\mRV[*]$. However, it is not hard to see that the final results also hold for other relevant categories.

Let $\mathbf{U} = (U,f,\omega) \in \mRV[n]$ and $\mathbf{V} = (V, g,\pi) \in \mRV[m]$. Let $\omega \times \pi : U \times V \fun \RV^{\times}$ be the function given by $(\lbar u, \lbar v) \efun \omega(\lbar u) \pi(\lbar v)$. The cartesian product $\mathbf{U} \times \mathbf{V} \in \mRV[n+m]$ is the object $((U,f) \times (V, g), \omega \times \pi)$. This makes $\gsk \mRV[*]$ into a graded semiring. Note that if $n=0$ or $m=0$ then we need to make some obvious adjustment on $\omega \times \pi$. Multiplication in $\gsk \mVF_*$ is defined analogously.

\begin{defn}\label{defn:blowup:vol}
An \emph{elementary blowup} of an object $(\mathbf{U}, \omega) \in \mRV[k]$ is an object $(\mathbf{U}, \omega)^{\sharp} = (\mathbf{U}^{\sharp}, \omega^{\sharp}) \in \mRV[k]$ such that $\mathbf{U}^{\sharp}$ is an elementary $\RV[k]$-blowup of $\mathbf{U}$ and $\omega^{\sharp}(\lbar t, s) = \omega(\lbar t)$.

Now the other notions in Definition~\ref{defn:blowup} formulated for $\RV[k]$ may be similarly formulated for $\mRV[k]$.
\end{defn}

\begin{lem}\label{elementary:blowups:preserves:iso:vol}
Let $(U, f, \omega), (V, g, \pi) \in \mRV[k]$ and $(U, f, \omega)^{\sharp}$, $(V, g, \pi)^{\sharp}$ two elementary blowups. If $[(U, f, \omega)] = [(V, g, \pi)]$ then $[(U, f, \omega)^{\sharp}] = [(V, g, \pi)^{\sharp}]$.
\end{lem}
\begin{proof}
The argument here is very similar to the one in the proof of~\cite[Theorem~9.16]{Yin:special:trans}. Let $F : (U, f, \omega) \fun (V, g, \pi)$ be an isomorphism. Without loss of generality, we may assume that $\dom(F) = U$, the dimensions of $U$, $V$ are $k -1$, and $\pr_{< k} \rest f(U)$, $\pr_{< k} \rest g(V)$ are finite-to-one. Then $F$ is also an isomorphism between $(U, \pr_{< k} \circ f)$ and $(V, \pr_{< k} \circ g)$ and hence, for almost all $\lbar u \in U$, $\jcb_{\RV} F ((\pr_{< k} \circ f)(\lbar u), \lbar u)$ is defined.

For each $\lbar u \in U$ with $\pr_k (f(\lbar u)) = r$ and $(\pr_k \circ g \circ F)(\lbar u) = s$, let $F_{\lbar u} : \RV^{>1} \fun \RV^{>1}$ be the $\lbar u$-definable function given by
\[
t \efun \bigl (\jcb_{\RV} F ((\pr_{< k} \circ f)(\lbar u), \lbar u) \bigr)^{-1} \cdot \frac{\omega(\lbar u)}{\pi(F(\lbar u))} \cdot \frac{r}{s} \cdot t
\]
if the right-hand side is defined; otherwise set $F_{\lbar u} = \id$. Note that $F_{\lbar u}$ is always a bijection. Let $F^* : (U, f)^{\sharp} \fun (V, g)^{\sharp}$ be the $\RV[k]$-isomorphism given by $(\lbar u, t) \efun (F(\lbar u), t)$ (see Lemma~\ref{elementary:blowups:preserves:iso}). Observe that, for almost all $(\lbar u, t) \in U^{\sharp}$, the $k \times k$ matrix used to produce $\jcb_{\RV} F^* (f^{\sharp}(\lbar u, t), \lbar u, t)$ is of the form $\bigl(\begin{smallmatrix}
\lambda_{\lbar u}& 0 \\ 0 & 1
\end{smallmatrix}\bigr)$, where $\lambda_{\lbar u}$ is the $(k-1) \times (k-1)$ matrix used to produce $\jcb_{\RV} F ((\pr_{< k} \circ f)(\lbar u), \lbar u)$. So an easy computation shows that the function $F^{\sharp} : U^{\sharp} \fun V^{\sharp}$ given by
\[
(\lbar u, t) \efun (\lbar u, F_{\lbar u}(t)) \efun (F(\lbar u), F_{\lbar u}(t)).
\]
is a $\mRV[k]$-isomorphism.
\end{proof}

\begin{defn}
Let $\misp[k]$ be the subclass of $\ob \mRV[k] \times \ob \mRV[k]$ of those pairs $(\mathbf{U}, \mathbf{V})$ such that there exist isomorphic iterated blowups $\mathbf{U}^{\sharp}$, $\mathbf{V}^{\sharp}$. Let $\misp[*] = \coprod_{0 \leq k} \misp[k]$.
\end{defn}

Since Corollary~\ref{eblowup:same:loci:iso} and Lemma~\ref{isp:transitive} only formally depend on Lemma~\ref{elementary:blowups:preserves:iso}, we can derive their analogs here from Lemma~\ref{elementary:blowups:preserves:iso:vol} and hence the following:

\begin{lem}\label{isp:congruence:vol}
$\misp[k]$ is a semigroup congruence relation and $\misp[*]$ is a semiring congruence relation.
\end{lem}


Let $(A, \omega) \in \mVF[k]$ and $A_{\omega} = \{(\lbar a, \omega(\lbar a)) : \lbar a \in A) \}$. For simplicity the volume form on $A_{\omega}$ that is naturally induced by $\omega$ is still denoted by $\omega$. Clearly $(A, \omega)$ and $(A_{\omega}, \omega)$ are isomorphic. If $\wh T_{\sigma}$ is a standard contraction of $A_{\omega}$ then $\omega$ naturally induces a volume form $\omega_{\wh T_{\sigma}}$ on $(\wh T_{\sigma}(A_{\omega}), \pr_{\leq k})$. The function $\wh T_{\sigma}$ --- or the object $(\wh T_{\sigma}(A_{\omega}), \pr_{\leq k}, \omega_{\wh T_{\sigma}}) \in \mRV[k]$, which is completely determined by $\wh T_{\sigma}$ --- is understood as a \emph{standard contraction} of $(A, \omega)$.

Now it is not hard to reproduce the other results in Section~\ref{section:kernel} for the current context. For that, Lemma~\ref{special:tran:vol:pre} and Lemma~\ref{jcb:chain} are constantly used. Occasionally we also need Corollary~\ref{lift:iso:mrv:iso} (in combination with Lemma~\ref{simul:special:dim:1}), for example, the analogs of Corollary~\ref{kernel:dim:1} and Lemma~\ref{isp:VF:fiberwise:contract:isp}.

It is rather straightforward to state and prove the analogs of the results from Lemma~\ref{special:bi:lh:1:equi:blowup} to Corollary~\ref{iterated:blowup:same:RV}. They are left to the reader.

\begin{lem}\label{isp:VF:fiberwise:contract:isp:vol}
Let $(A_1, \omega_1), (A_2, \omega_2) \in \mVF[k]$ such that $\pvf (A_1) = \pvf (A_2) = A$. Suppose that there is an $E \sub \mathds{N}$ such that, for every $\lbar a \in A$,
\[
([\fib(A_1, \lbar a), \omega_1]_{E}, [\fib(A_2, \lbar a), \omega_2]_{E}) \in \misp \la \lbar a \ra.
\]
Let $\wh T_{\sigma}$, $\wh R_{\sigma}$ be two standard contractions of $(A_1, \omega_1)$, $(A_2, \omega_2)$. Set $E' = E \cup I_k$. Then
\[
([\wh T_{\sigma}(A_1), \omega_{1,\wh T_{\sigma}}]_{E'}, [\wh R_{\sigma}(A_2), \omega_{2,\wh R_{\sigma}}]_{E'}) \in \misp.
\]
\end{lem}
\begin{proof}
We modify the special bijection $T_A$ in the proof of Lemma~\ref{isp:VF:fiberwise:contract:isp} so that for any $\rv$-polydisc $\gp \sub A^{\sharp}$ the volume forms on the fiber $U_{\gp, 1}$ are the same and similarly for $U_{\gp, 2}$. This implies in particular that the induced volume forms on $A_1^{\sharp}$, $A_2^{\sharp}$ are constant on $\rv$-polydiscs. Then the proof of Lemma~\ref{isp:VF:fiberwise:contract:isp} goes through here with virtually no changes (the computations involving $\jcb_{\RV}$ are all straightforward).
\end{proof}

Let $(A_1, \omega_1), (A_2, \omega_2) \in \mVF[k]$ and $f : A_1 \fun A_2$ a unary
bijection relative to the coordinate $i \in I_k$. An almost trivial computation shows that $f$ is measure-preserving if and only if $f \rest \fib(A_1, \lbar a)$ is measure-preserving for every $\lbar a \in \pr_{\widetilde i}(A_1)$.

\begin{cor}\label{contraction:same:perm:isp:vol}
Let $(A_1, \omega_1), (A_2, \omega_2) \in \mVF[k]$ and $f : A_1 \fun A_2$ a measure-preserving unary bijection relative to the coordinate $i \in I_k$. Then for any permutation $\sigma$ of $I_k$ with $\sigma(1) = i$ and any standard contractions $\wh T_{\sigma}$, $\wh R_{\sigma}$ of $(A_1, \omega_1)$, $(A_2, \omega_2)$,
\[
([\wh T_{\sigma}(A_1), \omega_{1,\wh T_{\sigma}}]_{\leq k}, [\wh R_{\sigma}(A_2), \omega_{2,\wh R_{\sigma}}]_{\leq k}) \in \misp.
\]
\end{cor}

\begin{lem}\label{contraction:perm:pair:isp:vol}
Let $(A, \omega) \in \mVF[k]$. Let $i, j \in I_k$ be distinct and $\sigma_1, \sigma_2$ two permutations of $I_k$ such that
\[
\sigma_1(1) = \sigma_2(2) = i, \quad \sigma_1(2) = \sigma_2(1) = j, \quad \sigma_1
\rest \set{3, \ldots, k} = \sigma_2 \rest \set{3, \ldots, k}.
\]
Then, for any standard contractions $\wh T_{\sigma_1}, \wh T_{\sigma_2}$ of $(A, \omega)$,
\[
([\wh T_{\sigma_1}(A), \omega_{\wh T_{\sigma_1}}]_{\leq k}, [\wh T_{\sigma_2}(A), \omega_{\wh T_{\sigma_2}}]_{\leq k}) \in \misp.
\]
\end{lem}
\begin{proof}
Let $B$ be as in the proof of Lemma~\ref{contraction:perm:pair:isp} and regard (the corresponding restriction of) $\omega$ as a volume form on $B$. Then, by Corollary~\ref{subset:partitioned:2:unit:contracted} and compactness, there is a definable bijection $f : B_{\omega} \fun \VF^2 \times \RV^l$, unary relative to both coordinates, such that $f(B_{\omega}) = B'$ is a $2$-cell. Clearly $f$ is measure-preserving with respect $\omega$ and the induced volume form $\pi$ on $B'$. So, by Corollary~\ref{contraction:same:perm:isp:vol} and Lemma~\ref{isp:VF:fiberwise:contract:isp:vol}, it is enough to find two standard contractions $\wh R_{ij}$, $\wh R_{ji}$ of $B'$ such that
\[
[\wh R_{ij}(B'), \pi_{\wh R_{ij}}]_{\leq 2} = [\wh R_{ji}(B'), \pi_{\wh R_{ji}}]_{\leq 2}.
\]
This would follow if we have a version of Corollary~\ref{subset:partitioned:2:unit:contracted} for objects in $\mVF[2]$, which would follow from a version of Lemma~\ref{2:unit:contracted} for $2$-cells with volume forms. To that end, let us assume $B' \sub \VF^2$. By inspection of the proof of Lemma~\ref{2:unit:contracted}, we see that it is enough to consider the following two cases: $B'$ is a product of two open balls or the positioning function $\epsilon$ is balanced in $B'$. The first case is obvious since we can simply use the identity map. In the second case, observe that there is no requirement on $\jcb_{\RV}$ since the resulting contractions are of dimension $1$. On the other hand, the requirement on $\jcb_{\Gamma}$ is satisfied by Remark~\ref{rem:rv:linear:bal}.
\end{proof}

Since the unary bijections in the proof of Lemma~\ref{bijection:partitioned:unary} are all given by standard contractions (additive translations), their Jacobians are constantly $1$ and hence they are measure-preserving with respect to the induced volume forms (from either side of the composition). Now we have reproduced for $\mVF[k]$, $\mRV[k]$ all the results that the proof of Proposition~\ref{kernel:L} formally depends on, so the following crucial description of the kernel of $\bb L$ may be obtained by the same proof:

\begin{prop}\label{kernel:L:vol}
$[\bb L \mathbf{U}] = [\bb L \mathbf{V}]$ if and only if $([\mathbf{U}], [\mathbf{V}]) \in \misp$.
\end{prop}

The derivations of the results from Proposition~\ref{main:prop:k} to Theorem~\ref{theorem:integration:1} are also of a formal nature, so at this point their analogs are immediate. 

\begin{prop}\label{main:prop:k:vol}
For each $k \geq 0$ there is a canonical isomorphism of Grothendieck semigroups
\[
\Xint\infty_{+} : \gsk \mVF[k] \fun \gsk \mRV[k]/\misp
\]
such that
\[
\Xint\infty_{+} [\mathbf{A}] = [\mathbf{U}]/\misp \quad \text{if and only if} \quad  [\mathbf{A}] = [\bb L \mathbf{U}].
\]
\end{prop}

Note that there is no need to weed out the element $\infty$ from $\mRV[k]$ and define a category $\mRV^{\times}[k]$ as in the last section, since $\infty$ is already ignored in $\mRV[k]$.

\begin{thm}\label{theorem:semiring:vol}
There is a canonical isomorphism of Grothendieck semirings
\[
\int_{+} : \gsk \mVF_* \fun \gsk \mRV[*]/\misp
\]
such that
\[
\int_{+} [\mathbf{A}] = [\mathbf{U}]/\misp \quad \text{if and only if} \quad  [\mathbf{A}] = [\bb L \mathbf{U}].
\]
\end{thm}

Let us write $\gsk \mVF_*$, $\gsk \mRV[*]$ simply as $\gsk \mVF$, $\gsk \mRV$ respectively. Let $A \sub \VF^n$ and $f : A \fun \mdl P(\RV^m)$ a definable function. Let $\omega$ be a volume form on $A_f = \bigcup_{\lbar a \in A}(\{\lbar a\} \times f(\lbar a))$. For each $\lbar a \in A$, $\omega_{\lbar a} = \omega \rest \fib(A_f, \lbar a)$ is understood as a volume form on $f(\lbar a)$ and hence $(f(\lbar a), \omega_{\lbar a})$ is an object in $\mRV[*]$. We shall write $(A_f, \omega)$ as $(f, \omega)$. As in Remark~\ref{remark:fn:semimodule}, $(f, \omega)$ is a representative of a \emph{definable function} $ A \fun \gsk \mRV/\misp$, which, for simplicity, is also denoted by $(f, \omega)$. The set of all such definable functions is denoted by $\fn(A, \gsk \mRV/\misp)$. Note that there is a subtle difference between the case without volume forms and the case with volume forms: $(f, \omega)$, $(g, \sigma)$ are the same function on $A$ if $([f(\lbar a)],[g(\lbar a)]) \in \misp \la \lbar a \ra$ for every $\lbar a \in A$ outside a definable subset of $A$ of $\VF$-dimension $< n$. Nevertheless, we may still define the integral of $(f,  \omega)$ by setting
\[
\int_{+A} ( f,  \omega) = \int_{+} [(f, \omega)] = \int_{+} [(\bb L f, \bb L \omega)],
\]
where $\bb L \omega$ is the volume form on $\bb L f$ naturally induced by $\omega$. By Proposition~\ref{kernel:L:vol} and compactness, this definition does not depend on the representative $(f, \omega)$. We have a homomorphism of $\gsk \mRV/\misp$-semimodules:
\[
\int_{+A} : \fn(A, \gsk \mRV/\misp) \fun \gsk \mRV/\misp.
\]

\begin{prop}\label{semi:fubini}
For any nonempty subsets $E_1, E_2 \sub I_n$,
\[
\int_{+\lbar a \in \pr_{E_1}(A)} \int_{+ \fib(A, \lbar a)} ( f,  \omega) = \int_{+\lbar a \in \pr_{E_2}(A)} \int_{+ \fib(A, \lbar a)} ( f,  \omega).
\]
\end{prop}

Let $B \sub \VF^n$ and assume that the dimensions of $A$, $B$ are $n$. Let $\phi : A \fun B$ be a definable bijection. Clearly the induced function
\[
\phi_* : \fn(A, \gsk \mRV/\misp) \fun \fn(B, \gsk \mRV/\misp)
\]
given by $( f,  \omega) \efun ( f \circ \phi^{-1}, \omega \circ \phi^{-1})$ is an isomorphism of $\gsk \mRV/\misp$-semimodules. However, this in general does not preserve integrals.
Let $\jcb_{\VF} \phi^{-1} \cdot \omega$ be the volume form on $B_{f \circ \phi^{-1}}$ given by
\[
(\lbar b, \lbar t) \efun \jcb_{\VF} \phi^{-1}(\lbar b) \cdot \omega_{\phi^{-1}(\lbar b)}(\lbar t)
\]
if this is defined; otherwise $(\lbar b, \lbar t) \efun 1$. The \emph{Jacobian transformation} of $( f,  \omega)$ with respect to $\phi$ is given by
\[
\phi^{\jcb}( f,  \omega) = ( f \circ \phi^{-1}, \jcb_{\VF} \phi^{-1} \cdot \omega),
\]
which does not depend on the representative $(f, \omega)$.

\begin{prop}\label{semi:change:variables}
$\int_{+A} ( f, \omega) = \int_{+B} \phi^{\jcb}( f,  \omega)$.
\end{prop}
\begin{proof}
Let $F : \bb L f \fun \bb L (f \circ \phi^{-1})$ be the bijection naturally induced by $\phi$. It is clear that, for almost all $(\lbar a, \lbar b, \lbar t) \in \bb L f$, $\jcb_{\VF} F(\lbar a, \lbar b, \lbar t) = \jcb_{\VF} \phi(\lbar a)$. Hence, by the definition of $\phi^{\jcb}$, we see that $(\bb L f, \bb L \omega)$, $(\bb L(f \circ \phi^{-1}), \bb L (\jcb_{\VF} \phi^{-1} \cdot \omega))$ are isomorphic.
\end{proof}

Let $\bf I$ be the corresponding ideal of the groupification $\ggk \mRV / \bf I$ of $\gsk \mRV /\misp$, which is determined by $\misp$ and is clearly homogeneous. Let $\mathbf{I}_k = \ggk \mRV[k] \cap \mathbf{I}$. Let $[1]_1 \in \gsk \mRV[1]$ be the isomorphism class of $(\{1\}, \id, 1)$ and $[\RV^{>1}]_1 \in \gsk \mRV[1]$ the isomorphism class of $(\RV^{>1}, \id, 1)$. Clearly $([1]_1, [\RV^{>1}]_1) \in \misp$. Let $[(U, f, \omega)] \in \gsk \mRV[k]$ and $(U, f, \omega)^{\sharp}$ an elementary blowup of $(U, f, \omega)$. Let
\[
U' = U \times \{\infty\}, \quad f' = \bb E_{k-1}^{-1}(f^{\sharp} \rest U'), \quad \omega' = \omega \cdot f_k.
\]
We have
\[
[(U', f', \omega')] \times  [1]_1 = [(U, f, \omega)], \quad
[(U', f', \omega')] \times [\RV^{>1}]_1 = [(U, f, \omega)^{\sharp}].
\]
This shows that, as a semiring congruence relation on $\gsk \mRV$, $\misp$ is generated by $([1]_1, [\RV^{>1}]_1)$ and hence its corresponding ideal
$\bf I$ in the graded ring $\ggk \mRV$ is homogeneous and is
generated by the element $[1]_1 - [\RV^{>1}]_1$.

\begin{thm}\label{theorem:integration:1:vol}
The Grothendieck semiring isomorphism $\int_{+}$ induces canonically a ring isomorphism
\[
\int : \ggk \mVF \fun \ggk \mRV / \mathbf{I} =\bigoplus_{k \geq 0} \ggk \mRV[k] / \mathbf{I}_k.
\]
\end{thm}

In applications it is often more desirable to work without filtration. One of the simplest ways to do this is as follows. Let $\mathbf{l} \in \ggk \mRV / \mathbf{I}$ be the element $[1]_1 + \mathbf{I}$ and write $(\ggk \mRV / \mathbf{I})[\mathbf{l}^{-1}]$ as $\KR$. For each $k \in \Z$ let $\KR_k$ be the $k$th graded piece of $\KR$. Then we have a canonical semiring homomorphism
\[
\gsk \mRV/\misp \to \ggk \mRV / \mathbf{I} \epi \KR_0.
\]
Note that the second arrow in this composition is a surjective ring homomorphism but is not injective. Set
\[
\fn(A, \KR_0) =  \fn(A, \gsk \mRV/\misp) \otimes_{\gsk \mRV/\misp} \KR_0.
\]
Then $\int_{+A}$ induces a canonical homomorphism of $\KR_0$-modules:
\[
\int_{A} : \fn(A, \KR_0) \fun \KR_0.
\]
By Proposition~\ref{semi:fubini} we immediately have

\begin{thm}[Fubini's theorem]\label{fubini}
For any $\mathbf{f} \in \fn(A, \KR_0)$ and any nonempty subsets $E_1, E_2 \sub I_n$,
\[
\int_{\lbar a \in \pr_{E_1}(A)} \int_{\fib(A, \lbar a)} \mathbf{f} = \int_{\lbar a \in \pr_{E_2}(A)} \int_{\fib(A, \lbar a)} \mathbf{f}.
\]
\end{thm}

Let $B \sub \VF^n$ and $\phi : A \fun B$ a definable bijection. The Jacobian transformation $\phi^{\jcb}$ from $\fn(A, \KR_0)$ to $\fn(B, \KR_0)$ is defined analogously. By Proposition~\ref{semi:change:variables} we immediately have

\begin{thm}[Change of variables]\label{change:variables}
For any $\mathbf{f} \in \fn(A, \KR_0)$,
\[
\int_{A} \mathbf{f} = \int_{B} \phi^{\jcb} (\mathbf{f}).
\]
\end{thm}

\end{document}